\documentclass{amsart}
\usepackage{amsmath,amsthm, amscd, amssymb, amsfonts} 
\usepackage{mathrsfs}
\usepackage[english]{babel}
\usepackage[latin1]{inputenc}
\usepackage{graphicx}
\usepackage{enumerate}
\usepackage{esint}
\usepackage{color}
\usepackage{hyperref}
\usepackage{cleveref}
\usepackage[shortlabels]{enumitem}

\newtheorem{thm}{Theorem}
\newtheorem{cor}{Corollary}
\newtheorem{lem}{Lemma}
\newtheorem{prop}{Proposition}

\theoremstyle{remark}
\newtheorem{rem}{Remark}

\newcommand{\al}{\alpha}
\newcommand{\sig}{\sigma}

\newcommand{\de}{\delta}

\newcommand{\ga}{\gamma}

\newcommand{\be}{\beta}
\newcommand{\la}{\lambda}

\newcommand{\RR}{\mathbb{R}}
\newcommand{\MM}{\mathcal{M}}

\newcommand{\NN}{\mathbb{N}}

\newcommand{\Mloc}{M^{\text{\upshape loc}}}

\title[Weighted inequalities of Fefferman-Stein type...]{Weighted inequalities of Fefferman-Stein type for Riesz-Schr\"odinger transforms}

\author{B. Bongioanni, E. Harboure and P. Quijano}



\email{bbongio@santafe-conicet.gov.ar}
\address{Instituto de Matem\'atica Aplicada del Litoral CONICET-UNL, and Facultad de Ingenier\'ia y Ciencias H\'idricas, UNL. Colectora Ruta Nac. N 168, Paraje El Pozo.}
\email{harbour@santafe-conicet.gov.ar}
\address{Instituto de Matem\'atica Aplicada del Litoral, CONICET-UNL, and Facultad de Ingenier\'ia Qu\'imica, UNL. Colectora Ruta Nac. N 168, Paraje El Pozo.
	3000 Santa Fe, Argentina.}
\email{pquijano@santafe-conicet.gov.ar}
\address{Instituto de Matem\'atica Aplicada del Litoral, CONICET-UNL, and Facultad de Ingenier\'ia Qu\'imica, UNL. Colectora Ruta Nac. N 168, Paraje El Pozo.
	3000 Santa Fe, Argentina.}
\date{}

\begin{document}

\begin{abstract}
In this work we are concerned with Fefferman-Stein type inequalities. More precisely, given an operator $T$ and some $p$, $1<p<\infty$, we look for operators $\mathcal{M}$ such that the inequality
$$\int |Tf|^pw\leq C\int |f|^p \mathcal{M}w$$
holds true for any weight $w$. Specifically, we are interested in the case of $T$ being any first or second order Riesz transform associated to the Schr\"odinger operator $L=-\Delta + V$, with $V$ a non-negative function satisfying an appropriate reverse-H\"older condition. For the Riesz-Schr\"odinger transforms $\nabla L^{-1/2}$ and $\nabla^2 L^{-1}$ we make use of a result due to C. P\'erez where this problem is solved for classical Calder\'on-Zygmund operators.
\end{abstract}

\maketitle

\section{Introduction}
In the theory of weighted $L^p$-inequalities a relevant is the following: given an operator $T$ and $1<p<\infty$, to find a positive operator $\MM$ such that inequalities of the form
\begin{equation}\label{eq-general}
\int |Tf|^p w\ \leq\ \int |f|^p \MM w,
\end{equation}
hold for some reasonable set of functions $f$ of $\RR^d$, $d\ge 1$, and a general weight $w$, i.e.\ $w\in L^1_{\text{loc}}(\RR^d)$, $w\geq 0$. However, the above inequality become more interesting when $\MM w$ is finite a.e.\ and to that end it is desirable to get the operator $\MM$ as small as possible.

The first appearance of such inequality goes back to the classical result of Fefferman-Stein (\cite{FeffermanSteinMR0284802}) for $T=\MM=M$, the Hardy-Littlewood maximal operator, namely
\begin{equation*}
\int_{\RR^d} |Mf|^p w\ \leq\ \int_{\RR^d} |f|^p M w,
\end{equation*}
for $1<p<\infty$.

When $T$ is a singular integral operator, C\'ordoba and Fefferman showed in \cite{Cordoba-Fefferman-MR0420115} that inequality~\ref{eq-general} holds taking $\MM= M_r = (M(w^r))^{1/r}$, for any $1<r<\infty$. However, it is known that for the Hilbert transform that inequality fails for $r=1$.

Later, Wilson in~\cite{wilson-weightednormMR972707} obtained inequalities for $1<p<2$ and $\MM=M\circ M$ improving the result in~\cite{Cordoba-Fefferman-MR0420115} since $M\circ M (w) \leq (M(w^r))^{1/r}$, for all $r>1$.

In 1995, C. P\'erez provided a full answer to this question with different techniques including weak type inequalities for $p=1$. He deals with maximal operators associated to averages with respect to a Young function  which can be smaller than $M_r$.

Below, we state the precise statements since they are essential to our work.

By a Young function $A$ we mean $A:[0,\infty)\rightarrow [0,\infty)$ continuous, convex, increasing and such that $A(0)=0$.  To define a maximal operator associated to a Young function $A$ we introduce the $A$-average of a function $f$ over a ball $B$ as
\begin{equation}
\|f\|_{A,B}=\inf\left\{\lambda>0:\frac{1}{|B|}\int_B A\left(\frac{|f(t)|}{\lambda}\right)dt\leq 1\right\}.
\end{equation}
Then, the maximal operator associated to a Young function $A$ is
\begin{equation}
M_A f(x)= \sup_{B\ni x} \|f\|_{A,B}.
\end{equation}
For $1<p<\infty$, we define  $\mathcal{D}_p$ as the class of Young functions such that 
\begin{equation}\label{eq-claseA}
\int_c^{\infty}\left(\frac{t}{A(t)}\right)^{p'-1}\frac{dt}{t}<\infty
\end{equation}
for some $c>0$.

The following theorem appears as Theorem~1.5 in \cite{PerezWeightedNormMR1260114}. There it is stated for singular integral operators. But according to the comment in Section~3 there, it also holds for Calder\'on-Zygmund operators as it is stated next.

\begin{thm}\label{teo-perez-tipofuerte}
	Let $1<p<\infty$, and let $T$ be a Calder\'on-Zygmund operator. Suppose that $A\in \mathcal{D}_p$. 
	 Then there exists a constant $C$ such that for each weight $w$
	\begin{equation}
	\int |Tf|^pw \leq C\int |f|^p M_A w.
	\end{equation}
\end{thm}

The following theorem deals with the endpoint case $p=1$ and it is also due to C. P\'erez. Here we state a version that can be found in~\cite{Cruz-Uribe_libro_MR2797562} as Theorem~$9.31$.
\begin{thm}\label{teo-perez-tipodebil}
	Let $T$ be a Calder\'on-Zygmund operator and let $A\in \bigcup_{p>1} \mathcal{D}_p$. 
	Then there exists a constant $C$ such that for each weight $w$ and for all $\lambda>0$ we have
	\begin{equation}
	w(\{y\in\RR^d:|Tf(y)|>\lambda\})\leq \frac{C}{\lambda}\int |f(y)| M_{A} w(y) dy.
	\end{equation}
\end{thm}

Some examples of functions on the class $\mathcal{D}_p$ are $A(t)=t \log^{p-1+\varepsilon}(1+t)$ or $A(t)=t \log^{p-1}(1+t)\log^{p-1+\varepsilon}(\log(1+t))$ for any $\varepsilon>0$. As for the class $\bigcup_{p>1}\mathcal{D}_p$, we can take $A(t)=t\log^{\varepsilon}(1+t)$ for any $\varepsilon>0$.

In this work we attempt to provide results of this type for the first and second order Riesz transforms associated to the Schr\"odinger differential operator $L=-\Delta~+~V$ on $\RR^d$, $d\ge 3$ and with $V$ satisfying a reverse H\"older inequality of order $q$, $q>d/2$, that is, there exists $C$ such that

\begin{equation}\label{eq-RH}
\left(\frac{1}{|B|}\int_B V^q\right)^{{1}/{q}} \leq
C \frac{1}{|B|}\int_B V,
\end{equation}
holds for every ball $B$ in $\RR^d$. From now on, if a function $V$ satisfy~\eqref{eq-RH} above we will say that $V\in RH_q$.

The study of these operators under such assumptions on $V$, was started by Shen in~\cite{shen}, where he proves $L^p$ boundedness for most of the operators we will be concerned with. As he observed, when $q>d$, the first order Riesz transforms $\nabla L^{-1/2}$ are standard Calder\'on-Zygmund operators. Otherwise, they are not necessarily bounded on $L^p$ for all $p$, $1<p<\infty$. The case of the second order Riesz transforms given by $\nabla^2L^{-1}$ is even worse since one can assures boundedness only for $1<p<q$. However, we may expect in inequality \eqref{eq-general} a smaller operator $\MM$ than those given by P\'erez, since Schr\"odinger Riesz transforms have kernels with a better decay at infinity. Also, in this context, kernels may have no symmetry and hence we might obtain different results for $T$ and its adjoint.

Essentially, we will be considering two types of first and second order Riesz transforms: one involving only derivatives $\nabla L^{-1/2}$ and $\nabla^2 L^{-1}$, and the others involving the potential $V$, as $V^{1/2}L^{-1/2}$, $V L^{-1}$ and $V^{1/2}\nabla L^{-1}$. In the first case we will get our results by locally comparing with the classical Riesz transforms, allowing us to apply the results of C. P\'erez. Let us point out that for $\nabla L^{-1/2}$ such comparison estimate appeared already in~\cite{shen} but that is not the case for $\nabla^2L^{-1}$, so it must be provided. We do that in Lemma~\ref{lem-comparacionR2} and we believe it might be useful for other purposes. As for those operators involving $V$ we shall require only estimates on the size of their kernels. 

We would like to make a remark about the values of $p$ for which inequalities like \eqref{eq-general} will be obtained. In all instances the operator $\MM$ on the right hand side satisfies $\MM(1)\leq 1$ and therefore our results would imply boundedness on $L^p$, so the range of $p$ should be limited as in the original work of Shen.

The paper is organized as follows. In the next section we state some general theorems in a somehow abstract framework but having in mind the Schr\"odinger Riesz transforms mentioned above, leaving all the proofs and technical lemmas to Section~\ref{sec-proofs}

The results include strong type $(p,p)$ inequalities like \eqref{eq-general} as well as weak type $(1,1)$ estimates for a suitable class of operators and their adjoints. Let us remark that inequalities for the adjoint operators are not obtained by duality. In fact, if we proceed in that way we would not arrive to an inequality with an arbitrary weight on the left hand side as we wanted.

Section~\ref{sec-aplSchrodinger} is devoted to apply the general theorems of Section~\ref{sec-general} to specific operators associated to Schr\"odinger semigroup: $\nabla L^{-1/2}$, $\nabla^2L^{-1}$, $V^{\al}L^{-\al}$, $V^{\al-1/2}\nabla L^{-\al}$, with $\al$ in a range depending on the operator. In order to check that their kernels satisfy the required assumptions, sometimes we make use of known estimates by in other occasions we must prove them. In particular we prove a local comparison between the kernels of $\nabla^2 (-\Delta)^{-1}$ and $\nabla^2 L^{-1}$ stated in Lemma~\ref{lem-comparacionR2}.

Finally in the last section we use the above results to get sufficient conditions on a function $f$ to ensure local  integrability of $Tf$, where $T$ is any of the operators of Section~\ref{sec-aplSchrodinger}. Consequently we obtain a large class of functions $f$ such that $Tf$ is finite a.e.. In fact, $f$ is allowed to increase polynomially.  When this results are applied to the Riesz-Schr\"odinger transforms they provide qualitative information about solutions of some differential equations involving $L$.

\section{General Results}\label{sec-general}
In this section we will consider the space $\RR^d$ equipped with a \textit{critical radius function} $\rho:\RR^d\rightarrow (0,\infty)$, that is, a function whose variation is controlled by the existence of $C_0$ and $N_0\geq 1$ such that
	\begin{equation} \label{eq-constantesRho}
	C_0^{-1}\rho(x) \left(1+ \frac{|x-y|}{\rho(x)}\right)^{-N_0}
	\leq \rho(y)
	\leq C_0 \rho(x) \left(1+ \frac{|x-y|}{\rho(x)}\right)^{\tfrac{N_0}{N_0+1}}.
	\end{equation}
	It is worth noting that if  $\rho$ is a critical radius function , then for any $\gamma>0$ the function $\gamma\rho$ is also a critical radius function. Moreover, if $0<\ga\leq 1$ then $\gamma\rho$ satisfies~\eqref{eq-constantesRho} with the same constants as $\rho$.

	The next lemma is a useful consequence of the previous inequality. 
	\begin{lem}[see~\cite{BHScommut-w_JFAA}, Corollary~1 ]\label{lem-trucho}
		Let $x$, $y\in B(x_0,R_0)$. Then:
		\begin{enumerate}[$i)$]
			\item There exists $C>0$ such that
			$$1 + \frac{R_0}{\rho(y)}\leq C\left(1+\frac{R_0}{\rho(x_0)}\right)^{N_0}.$$
			\item There exists $C>0$ such that
			$$1+\frac{r}{\rho(y)}\leq C\left(1+\frac{R_0}{\rho(x_0)}\right)^{\gamma}
			\left(1+\frac{r}{\rho(x)}\right),$$
			for all $r>R_0$, where $\gamma=N_0\left(1+\frac{N_0}{N_0+1}\right)$.
		\end{enumerate}
	\end{lem}


Associated to a critical radius function $\rho$ we can define the following maximal operators. First, let us denote $\mathcal{F}_{\rho}$ the set of all balls $B(x,r)$ such that $r\leq \rho(x)$. Then, for $f$ a locally integrable function, and $A$ a Young function, we set
\begin{equation}
\Mloc_A f(x)=\sup_{\substack{B\ni x\\ B\in \mathcal{F}_{\rho}}} \|f\|_{A,B},
\end{equation}
and for $\theta\geq 0$,
\begin{equation}
M^{\theta}_A f(x)=\sup_{B(x_0,r_0)\ni x} \left(1+\frac{r_0}{\rho(x_0)}\right)^{-\theta} \|f\|_{A,B}.
\end{equation}
As usual, when $A(t)=t^r$ we use the notation $\Mloc_r$ and $M^{\theta}_r$ respectively.

Now, we are in position to state our main theorems.

\begin{thm}\label{teo-prin}
	
	Let $T$ be a linear operator with associated kernel $K$. Suppose that for some $s>1$, $K$ satisfies the following estimates
	\begin{enumerate}[$(a_s)$]
		\item\label{cond-a} For each $N>0$ there exists $C_N$ such that 
		\begin{equation*}
		\left(\int_{R<|x_0-x|<2R}|K(x,y)|^s dx\right)^{1/s}
		\leq C_N R^{-d/s'} \left(1+ \frac{R}{\rho(x_0)}\right)^{-N},
		\end{equation*}
		whenever $|y-x_0|<R/2$.
		\item\label{cond-b} There exists a Calder\'on-Zygmund operator $T_0$ with kernel $K_0$ such that, for some $C$ and $\delta>0$,
		\begin{equation*}
		\left(\int_{R<|x_0-x|<2R}|K(x,y) - K_0(x,y)|^s dx\right)^{1/s}
		\leq C R^{-d/s'} \left(\frac{R}{\rho(x_0)}\right)^{\delta},
		\end{equation*}
		whenever $|y-x_0|<R/2$ with $R\leq \rho(x_0)$.
	\end{enumerate}
	Then, for each $\theta\geq 0$, the operator $T$ and its adjoint $T^{\star}$ satisfy the following inequalities for any weight $w$,
	\begin{equation}\label{eq-desigualdad-derechas}
	\int |Tf|^p w \leq C_{\theta} \int |f|^p M_r^{\theta}w,
	\end{equation}
	for $1<p<s$ and $r=(s/p)'$,
	
	\begin{equation}\label{eq-desigualdad-adjuntas}
	\int |T^{\star}f|^{p} w \leq C_{\theta} \int |f|^{p} (\Mloc_A + M^{\theta})w,
	\end{equation}
	for $s'<p<\infty$ and any Young function $A\in\mathcal{D}_p$. 
\end{thm}

\begin{rem}
	Assumption~\ref{cond-a} can be seen as a size condition with a kind of  ``decay at infinity'', while condition~\ref{cond-b} tells us that $K$ has the same singularity as a Calder\'on-Zygmund kernel. Nevertheless, both conditions on $K$ are not symmetric since integration is always made in the first variable. Consequently we do not get the same kind of estimates for $T$ and $T^{\star}$.
\end{rem}

If the kernel $K$ satisfies point-wise estimates we obtain a sharper result for $T$, as a corollary of the previous theorem. 

\begin{cor}\label{cor-infty}
	Let $T$ be a linear operator with associated kernel $K$ and $T_0$ be a Calder\'on-Zygmund operator with kernel $K_0$. Suppose that $K$ satisfy the following estimates.
	\begin{enumerate}[$(a_{\infty})$] 
		\item\label{cond-a_infty} For each $N>0$ there exists $C_N$ such that 
		\begin{equation*}
		|K(x,y)|
		\leq \frac{C_N}{|x-y|^d}  \left(1+ \frac{|x-y|}{\rho(x)}\right)^{-N}.
		\end{equation*}
		\item\label{cond-b_infty} There exist $C$ and $\delta>0$ such that 
		\begin{equation*}
		|K(x,y) - K_0(x,y)|
		\leq \frac{C}{|x-y|^d} \left(\frac{|x-y|}{\rho(y)}\right)^{\delta}.
		\end{equation*}
	\end{enumerate}
	Then, $T$ and its adjoint $T^{\star}$ satisfy~\eqref{eq-desigualdad-adjuntas}  for $1<p<\infty$ and  any Young function $A\in\mathcal{D}_p$.

\end{cor}

Corollary~\ref{cor-infty} follows inmediately from Theorem~\ref{teo-prin} since conditions~\ref{cond-a_infty} and~\ref{cond-b_infty} imply conditions~\ref{cond-a} and~\ref{cond-b} for all $1<s<\infty$, and are symmetric in $x$ and $y$.
For the limiting case $p=1$ we can obtain the following weak-type inequalities.

\begin{thm}\label{teo-tipo debil}
	Let $T$ be a linear operator with associated kernel $K$ and let $T_0$ be a Calder\'on Zygmund operator with kernel $K_0$. Suppose that for some $s>1$, $K$ satisfies conditions~\ref{cond-a} and~\ref{cond-b}	
	Then, for  $\theta\geq 0$ and $w\in L^1_{loc}$, $w\geq 0$, T satisfies
	\begin{equation}\label{eq-desigualdad-tipodebil}
	w(\{|Tf|>\lambda \})\leq \frac{C_{\theta}}{\lambda} \int |f| M_{s'}^{\theta}(w), \text{    for   } \lambda>0.
	\end{equation}
	Further, if $T$ satisfies~\ref{cond-a_infty} and~\ref{cond-b_infty}, then, for any Young function $A\in\bigcup_{p>1}\mathcal{D}_p$, 
		
		\begin{equation}\label{eq-desigualdad-tipodebilbuena}
		w(\{|Tf|>\lambda \})\leq \frac{C_{\theta}}{\lambda} \int |f| \left(\Mloc_{A}+ M^{\theta}\right)w,\text{    for   } \lambda>0.
		\end{equation}
	 Moreover, inequality~\eqref{eq-desigualdad-tipodebilbuena} also holds for $T^{\star}$.
\end{thm}

The associated kernels of some operators related to $L$ satisfy condition~\ref{cond-b} without subtracting $K_0$ and hence condition $(a_s)$ and $(b_s)$ can be unify. For this type of operators we can get sharper inequalities stated in the following theorem.

\begin{thm} \label{teo-deigualdades-operad-V}
	Let $T$ be a linear operator with associated kernel $K$. Suppose that for some $s>1$ and $\de>0$, $K$ satisfies the following condition:  
	\begin{enumerate}[$(c_s)$]
		\item\label{cond-c} For each $N>0$, there exists $C_N$ such that
		\begin{equation*}
		\left(\int_{R<|x_0-x|<2R}|K(x,y)|^s dx\right)^{1/s}
		\leq C_N R^{-d/s'} \left(1+\frac{\rho(x_0)}{R}\right)^{-\delta}
		\left(1+\frac{R}{\rho(x_0)}\right)^{-N},
		\end{equation*}
		whenever $|y-x_0|<R/2$.
	\end{enumerate}
	Then, for any $\theta\geq 0$ and any weight $w$, there exists $C_\theta$ such that $T$ satisfies~\eqref{eq-desigualdad-derechas} for $1\leq p<s$ and
	\begin{equation}\label{eq-desigualdad-adjuntas_conV}
	\int |T^{\star}f|^{p} w \leq C_{\theta} \int |f|^{p}  M^{\theta}w,
	\end{equation}
	for $s'<p<\infty$.
\end{thm}

\section{Proofs}\label{sec-proofs}
Before giving the proofs of the theorems above we need to state some technical lemmas that will be useful in the sequel. The first one is a consequence of inequality~\eqref{eq-constantesRho} and can be found in~\cite{DZ-HSH-99}.

In some proofs we will use the notation $\lesssim$ instead of $\leq$ to denote that the right hand side of the inequality is greater up to constants that may depend on some parameters specified when necessary.  

\begin{prop} \label{prop-cubrimientocritico}
	There exists a sequence of points $\{x_j\}_{j\in\NN}$ such that the family of critical balls $Q_j=B(x_j,\rho(x_j))$  satisfies
	\begin{enumerate}[i)]
		\item $\displaystyle \bigcup_{j\in\NN} Q_j= \RR^d$
		\item There exist constants $C$ and $N_1$ such that for any $\sigma\geq1$,
		$\displaystyle \sum_{j\in\NN}\chi_{\sigma Q_j}\leq C \sigma^{N_1}$.
	\end{enumerate}
\end{prop}

In general, maximal operators can not be controlled point-wisely by localized ones. Nevertheless, this is possible if we are considering functions supported on sub-critical balls and for points close enough to the support. In the next lemma we determine how much a critical ball must be contracted in order to have that kind of control. Such contraction of critical balls is needed to arrive to inequality \eqref{eq-desigualdad-adjuntas} of Theorem~\ref{teo-prin}.

\begin{lem}\label{lem-gamma0}
	Let $A$ be a Young function and $B_0$ any critical ball. There exists $\gamma_0>0$ such that if $0<\gamma\leq \gamma_0$ then for any function $f$,
	\begin{equation}
	M_A(f\chi_{\gamma B_0})(x) \leq C \Mloc_A (f)(x),
	\end{equation}
	for all $x\in 2\gamma B_0$. Here, the constant $C$ only depends on the dimension $d$ and the Young function $A$.
\end{lem}
\begin{proof}	
	Assume $x\in 2\ga B_0$ with $\ga$ to be determined later. It is enough to consider balls centered at $x$; in fact, it is not difficult to see that if $M_A^c$ is the centered maximal function, then $M_A(f)(x) \leq C M_A^c (f)(x)$ for any function $f$ with $C$ that only depends on $d$ and $A$. Let $x_0$ be the center of $B_0$ and suppose first that $r>3\ga\rho(x_0)$. Therefore $B(x,r)\supset B(x,3\ga\rho(x_0))\supset\ga B_0$ and thus, for any non-negative function $g$,
	\begin{equation*}
	\frac{1}{|B(x,r)|}\int_{B(x,r)\cap\ga B_0}g \ \leq \ \frac{1}{|B(x,3\ga\rho(x_0))}\int_{\ga B_0} g\leq \frac{1}{|B(x,3\ga\rho(x_0))|}\int_{B(x,3\ga\rho(x_0))|}g.
	\end{equation*}
	
	Now, if $\la>0$, applying the above inequality to $g=A(|f|/\la)$ we have, for $r\ge 3\ga\rho(x_0)$,
	\begin{equation*}
	\|f\chi_{\ga B_0}\|_{A,B(x,r)} \leq \|f\|_{A,B(x,3\ga B_0)}.
	\end{equation*}
	
	Therefore, if $x\in 2\ga B_0$,
	\begin{equation*}
	M_A^{c}(f\chi_{\ga B_0})(x) \ \leq \sup_{r\leq 3\ga\rho(x_0)}\|f\|_{A,B(x,r)}.
	\end{equation*}
	
	To complete the proof, it is enough to take $\ga$ such that $3\ga\rho(x_0)\leq \rho(x)$ for all $x\in2\ga B_0$.
	
	From inequality \eqref{eq-constantesRho}, we have $\rho(x_0)\leq\rho(x)C_0(1+2\ga)^{N_0}$ and thus $\ga$ should be taken such that
	\begin{equation}\label{eleccion_ga}
		3\ga C_0 (1+2\ga)^{N_0} \leq 1.
	\end{equation}
	Since the left hand side goes to $0$ when $\ga$ goes to $0$, there exists $\ga_0$ such that for $0<\ga\ge\ga_0$ the above inequality holds.
	
	\end{proof}

Conditions~\ref{cond-a} and~\ref{cond-b} are written in a suitable way to prove inequalities concerning $T^\star$. To prove the inequalities for $T$ it will be easier to use the following equivalent conditions.   
\begin{lem}\label{lem-condicionesequivalentes}
	For any $s>1$, conditions~\ref{cond-a} and~\ref{cond-b} are equivalent to, respectively, to the following conditions.
	 	\begin{enumerate}[$(a'_s)$]
	 		\item\label{cond-a'} For each $N>0$ there exists $C_N$ such that 
	 		\begin{equation*}
	 		\left(\int_{B(x_0,R/2)}|K(x,y)|^s dx\right)^{1/s}
	 		\leq C_N R^{-d/s'} \left(1+ \frac{R}{\rho(x_0)}\right)^{-N},
	 		\end{equation*}
	 		whenever $R<|y-x_0|<2R$.
	 		\item\label{cond-b'} There exist $C$ and $\varepsilon>0$ such that 
	 		\begin{equation*}
	 		\left(\int_{B(x_0,R/2)}|K(x,y) - K_0(x,y)|^s dx\right)^{1/s}
	 		\leq C R^{-d/s'} \left(\frac{R}{\rho(x_0)}\right)^{\varepsilon}
	 		\end{equation*}
	 		whenever $R<|y-x_0|<2R$ and 
	 		$R\leq \rho(x_0)$.
	 	\end{enumerate}
\end{lem}

\begin{rem}\label{rem:asconC}
	Observe that $(a_s)$ holds true replacing the ring, $R<|x-x_0|<2R$ with $R<|x-x_0|<CR$ for any constant $C>1$, with the constant $C_N$ depending on $C$. Similarly in $(a_s')$ the ring $R<|y-x_0|<2R$ may be replaced by $R<|y-x_0|<CR$. In fact, it is only a matter of applying $(a_s)$ or $(a_s')$ a finite number of times deppending on $C$.
	
	The same comment applies to $(b_s)$ and $(b_s')$.
\end{rem}

\begin{proof}[Proof of Lemma~\ref{lem-condicionesequivalentes}]
We will show first that \ref{cond-a} implies \ref{cond-a'}. Let $K$ be a kernel satisfying~\ref{cond-a} for some $s>1$, and let $x_0\in \RR^d$, $R>0$ and $y$ such that $R<|x_0-y|<2R$. It is easy to check that $B(x_0,R/2)\subset\{x: R/2<|x-y|<4R\}$. So, applying condition~\ref{cond-a} we get that
\begin{equation*}
\begin{split}
\left(\int_{B(x_0,R/2)}|K(x,y)|^{s}dx\right)^{1/s} 
& \leq \left(\int_{R/2<|y-x|<4R}|K(x,y)|^{s}dx\right)^{1/s} 
\\& \leq C_N R^{-d/s'}\left(1+\frac{R}{\rho(y)}\right)^{-N}
\\& \leq C_N R^{-d/s'}\left(1+\frac{R}{\rho(x_0)}\right)^{-\tilde{N}},
\end{split}
\end{equation*}
where in the last inequality we used Lemma~\ref{lem-trucho}.

To see that~\ref{cond-a'} implies \ref{cond-a} let $x_0\in \RR^d$, $R>0$ and $y\in B(x_0,R/2)$. The ring $\{x: R<|x-x_0|<2R\}$ can be covered by $M$ balls (depending on $d$), of radius $R/4$ and centres $x_i$, with $R<|x_i-x_0|<2R$, for $i=1,\dots, M$. For each of these balls we can check that $R/2<|x_i-y|<5R/2$. Applying condition~\ref{cond-a'} and Remark~\ref{rem:asconC} on each ball,
\begin{equation*}
\begin{split}
\left(\int_{R<|y-x|<2R}|K(x,y)|^{s}dx\right)^{1/s} 
& \leq \sum_{i=1}^{M}\left(\int_{B(x_i,R/4)}|K(x,y)|^{s}dx\right)^{1/s} 
\\& \leq  \sum_{i=1}^{M} C_N R^{-d/s'}\left(1+\frac{R}{\rho(x_i)}\right)^{-N}
\\& \leq C_N R^{-d/s'}\left(1+\frac{R}{\rho(x_0)}\right)^{-\tilde{N}},
\end{split}
\end{equation*}
where we used again Lemma~\ref{lem-trucho} in the last inequality.

We can omit the  proof of the equivalence of \ref{cond-b} and \ref{cond-b'} since it follows the same lines as above.

\end{proof}

\begin{proof}[Proof of Theorem~\ref{teo-prin}]
Let $T$ be a linear operator with kernel $K$ satisfying~\ref{cond-a} and~\ref{cond-b}, for some $s>1$ and some Calder\'on-Zygmund operator $T_0$ with kernel $K_0$.  Let $w\geq 0$, $w\in L^1_{loc}$, $\theta\geq 0$, $1<p<s$ and let A be a Young function satisfying~\eqref{eq-claseA}. 

We will prove first inequality~\eqref{eq-desigualdad-derechas}. Let $\gamma_0$ be as in Lemma~\ref{lem-gamma0}.  For some $\gamma\leq \gamma_0$, to be chosen later, let $\{Q_n\}$ be the decomposition of the space given in Proposition~\ref{prop-cubrimientocritico} for the critical radius function $\gamma\rho$. Then we write
\begin{equation}\label{eq-I+II+III}
\begin{split}
\int |Tf|^p w
& \le  \sum_{n\in\NN} \int_{Q_n} |Tf|^p w
\\& = \sum_{n\in\NN} \int_{Q_n} |T(f\chi_{2Q_n}) + T(f\chi_{2Q_n^c}) \pm T_0(f\chi_{2Q_n})|^p w
\\& \lesssim \sum_{n\in\NN} \int_{Q_n} |T(f\chi_{2Q_n}) - T_0(f\chi_{2Q_n})|^p w \
 + \ \sum_{n\in\NN} \int_{Q_n} |T(f\chi_{2Q_n^c})|^p w
\\ & \hspace{2cm} +\ \sum_{n\in\NN} \int_{Q_n} |T_0(f\chi_{2Q_n})|^p w =
I + II + III.
\end{split}
\end{equation}
 
 For $III$, since $T_0$ is a Calder\'on-Zygmund operator, we apply Theorem~\ref{teo-perez-tipofuerte} and Lemma~\ref{lem-gamma0} to get
 \begin{equation*}
 \begin{split}
  III &= \sum_{n\in\NN} \int |T_0(f\chi_{2Q_n})|^pw\chi_{Q_n}
  \\& \lesssim\sum_{n\in\NN} \int |f\chi_{2Q_n}| M_{A}(w\chi_{Q_n})
  \\& \lesssim \sum_{n\in\NN} \int_{2Q_n} |f|^p \Mloc_{A}w
  \\& \lesssim \int |f|^p \Mloc_{A}w,
 \end{split}
 \end{equation*}
 for any Young function $A\in \mathcal{D}_p$.
 
 For $k\in\mathbb{Z}$ we denote $Q_n^k=2^k Q_n$. To estimate $II$ we use Minkowski's and H\"older's inequalities to obtain
\begin{equation*}
\begin{split}
II & = \sum_{n\in\NN}\int_{Q_n} |T(f\chi_{(2Q_n)^c})|^p w
\\& = \sum_{n\in\NN}\int_{Q_n} \left[\int_{(2Q_n)^c} |K(x,y)||f(y)|dy\right]^p w(x) dx
\\& \leq \sum_{n\in\NN}\left[\int_{(2Q_n)^c}|f(y)| \left(\int_{Q_n}|K(x,y)|^p w(x)dx\right)^{1/p}dy \right]^p
\\& \leq \sum_{n\in\NN}\left[ \sum_{k\in\NN}\int_{Q_n^{k+1}\setminus Q_n^k}|f(y)| \left(\int_{Q_n}|K(x,y)|^sdx\right)^{1/s} \left(\int_{Q_n}w^r(x)dx\right)^{1/rp} dy\right]^p,
\end{split}
\end{equation*}
where $r=(s/p)'$.

Next we apply condition~\ref{cond-a'} for $K$, since by Lemma~\ref{lem-condicionesequivalentes} condition \ref{cond-a} is equivalent to \ref{cond-a'}, then for each $N$ we have
\begin{equation*}
\begin{split}
	II & \lesssim \sum_{n\in\NN}\left[ \sum_{k\in\NN} |Q_n^k|^{-1/s'} 2^{-kN}
	\int_{Q_n^k}|f(y)| 
	 \left(\int_{Q_n}w^r\right)^{1/rp} dy\right]^p
\\& \lesssim \sum_{n\in\NN}\left[ \sum_{k\in\NN} |Q_n^k|^{-1/s'+1/p'} 2^{-kN}
\left(\int_{Q_n^k}|f(y)|^p
\left(\int_{Q_n^k}w^r\right)^{1/r}dy\right)^{1/p}\right]^p
\\& \lesssim  \sum_{n\in\NN}\left[ \sum_{k\in\NN}  2^{-k(N-\theta/p)}
\left(\int_{Q_n^k}|f(y)|^p\,
2^{-k\theta}\left(\frac{1}{|Q_n^k|}\int_{Q_n^k}w^r\right)^{1/r}dy\right)^{1/p}\right]^p
\\& \lesssim  \sum_{n\in\NN}\left[ \sum_{k\in\NN}  2^{-k(N-\theta/p)}
\left(\int_{Q_n^k}|f(y)|^p\,
M_r^{\theta}w(y) dy\right)^{1/p}\right]^p,
\end{split}
\end{equation*}
with constants that may depend on $N$.

Finally, using H\"older's inequality in the sum over $k$ and choosing $N=N_1 -\theta/p+1$, where $N_1$ is the constant appearing in Proposition~\ref{prop-cubrimientocritico}, we arrive to
\begin{equation*}
\begin{split}
II & \lesssim \sum_{n\in\NN}\left[ \sum_{k\in\NN}  2^{-k(N_1+1)}
\int_{Q_n^k}|f|^p\,
M_r^{\theta}w\right] \left[ \sum_{k\in\NN}  2^{-k(N_1+1)}
\right]^{p/p'}
\\ & \lesssim\sum_{k\in\NN} 2^{-k(N_1+1)} 
\int_{\RR^d}\left(\sum_{n\in\NN} \chi_{Q_n^k}\right)
|f|^p\,
M_r^{\theta}w 
\\ & \lesssim
\int_{\RR^d}
|f|^p\,
M_r^{\theta}w,
\end{split}
\end{equation*}
with constants depending on $N_1$ and $\theta$ and $p$.

It only remains to estimate $I$. Let $Q(x,y)=K(x,y)-K_0(x,y)$. For $x\in Q_n$, we have $2Q_n\subset B(x,\rho(x))$ due to our choice of $\ga$ (see inequality \eqref{eleccion_ga}), therefore we may write
\begin{equation}\label{eq-pasar_a_h}
\begin{split}
I & = \sum_{n\in\NN} \int_{Q_n} |T(f\chi_{2Q_n})-T_0(f\chi_{2Q_n})|^p w
\\ & \leq \sum_{n\in\NN} \int_{Q_n}  \left[
\int_{2Q_n}|Q(x,y)||f(y)|dy
\right]^p w(x)dx
\\ & \leq \sum_{n\in\NN} \int_{Q_n}  \left[
\int_{B(x,\rho(x))}|Q(x,y)||f(y)|dy
\right]^p w(x)dx
\\ & \leq \int_{\RR^d}|h(x)|^p w(x)dx
= \|h\|^p_{L^p(\RR^d,w)},
\end{split}
\end{equation}
where
\begin{equation*}
h(x)=\int_{B(x,\rho(x))}|Q(x,y)||f(y)|dy.
\end{equation*}

For a fixed $k$ and for any $n$ we can take $2^{dk}$ disjoint balls of the form $B_{n,k}^l=B(x_{n,k}^l, 2^{-k}\gamma\rho(x_n))$ such that for $\sigma>\sqrt{d}$,
 $$\displaystyle Q_n\subset \bigcup_{l=1}^{2^dk} \sigma B_{n,k}^l\subset 2\sigma Q_n.$$
Moreover, there exists a constant depending only on $\sigma$ and $d$ such that,
\begin{equation*}
\sum_{l=1}^{2^{dk}} \chi_{\sigma B_{n,k}^l }\leq C_{d,\sigma}\chi_{2\sig Q_n}.
\end{equation*}
Therefore, from Proposition~\ref{prop-cubrimientocritico}, the family of balls $\{\sigma B_{n,k}^l\}_{l,n}$ covers $\RR^d$ and
$$\sum_{l,n} \chi_{\sigma B_{n,k}^l }\leq C_{d,\sigma,\rho}.$$

Let us fix $\sigma=2\sqrt{d}$. It is possible to choose $\gamma$ small enough such that if $x\in \sig B_{n,k}^l$ and $2^{-k-1}\rho(x)\leq|y-x|\leq 2^{-k}\rho(x)$ then 
$$y\in E_{n,k}^l=\{y:4\sqrt{d}\gamma2^{-k}\rho(x_n)\leq |y-x_{n,k}^l|\leq \beta \gamma2^{-k}\rho(x_n)\}.$$
for some constant $\be>4\sqrt{d}$ depending only on $\rho$ and $d$ \footnote{For example, it works taking $\ga=\frac{1}{2C_0(5\sqrt{d})^{N_0+1}}$ and $\be=2C_0^2(5\sqrt{d})^{N_0+2}$}.

 Now, we write $h$ in the following way
 \begin{equation*}
 h(x)=\sum_{k=0}^{\infty} h_k(x)=\sum_{k=0}^{\infty} \int_{B(x,2^{-k}\rho(x))\setminus B(x,2^{-k-1}\rho(x))}|Q(x,y)||f(y)|dy.
 \end{equation*}
 
 So, for this covering of the space described above, we may write
 \begin{equation}
 \begin{split}
 \|h_k\|_{L^p(w)}^p
  &\leq 
 \sum_{n,l}\int_{2\sqrt{d}{B}_{n,k}^l}
 \left[
 \int_{B(x,2^{-k}\rho(x))\setminus B(x,2^{-k-1}\rho(x))} |Q(x,y)||f(y)|dy
 \right]^p w(x)dx
 \\ & \leq 
 \sum_{n,l}\int_{2\sqrt{d}{B}_{n,k}^l}
 \left[
 \int_{E_{n,k}^l} |Q(x,y)||f(y)|dy
 \right]^p w(x)dx
  \\ & \leq 
  \sum_{n,l}\left[
  \int_{E_{n,k}^l}|f(y)|
  \left(\int_{2\sqrt{d}{B}_{n,k}^l}
   |Q(x,y)|^p w(x)dx\right)^{1/p}   
  dy\right]^p 
  \\& \leq 
  \sum_{n,l}\left[
  \int_{E_{n,k}^l}
  \!\!\!\!|f(y)|
  \left(\int_{2\sqrt{d}{B}_{n,k}^l}
  \!\!\!\!|Q(x,y)|^s dx\right)^{1/s} 
  \!\!\!\left(\int_{2\sqrt{d}{B}_{n,k}^l}\!\!\!\!w^r(x)dx\right)^{1/(rp)} 
  \!\!\!\!\!\!dy\right]^p ,
 \end{split}
 \end{equation}
 where we have used Minkowski's and H\"older's inequalities in the last two steps. Now, using condition~\ref{cond-b'} for $Q(x,y)$ (See Remark~\ref{rem:asconC}), we arrive to
 \begin{equation}
 \begin{split}
 \|h_k\|_{L^p(w)}^p
 & \lesssim \sum_{n,l}
 (2^{-k}\rho(x_n))^{-dp/s'} 2^{-k\delta p}
  \left[\int_{\beta B_{n,k}^l}
  \!\!\!\!|f(y)| 
 \left(\int_{\beta B_{n,k}^l}
 \!\!\!\! w^r(x)dx\right)^{1/(rp)} 
 dy\right]^p 
 \\ & \lesssim 2^{-k\delta p} \sum_{n,l}
 \int_{\beta B_{n,k}^l}|f(y)|^p 
 \left(\frac{1}{|\be B_{n,k}^l|}\int_{\beta B_{n,k}^l}w^r(x)dx\right)^{1/r} 
 dy 
 \\ & \lesssim 2^{-k\delta p} \sum_{n,l}
 \int_{\beta B_{n,k}^l}|f(y)|^p 
 M_r^{\theta}w(y)dy
 dy 
 \\ & \lesssim 2^{-k\delta p} \|f\|_{L^p(M_r^{\theta}w)}^p.
 \end{split}
 \end{equation}
 Finally,
 \begin{equation}\label{eq-de_hk_a_h}
 \begin{split}
 \|h\|_{L^p(w)} &\leq
 \sum_{k\geq 0} \|h_k\|_{L^p(w)}
 \lesssim\sum_{k\geq 0} 2^{-k\delta}\|f\|_{L^p(M_r^{\theta}w)}
 \lesssim \|f\|_{L^p(M_r^{\theta}w)}.
 \end{split}
 \end{equation}

Using the estimates obtained for $I$, $II$ and $III$ we arrive to inequality~\eqref{eq-desigualdad-derechas}.

Now, let us prove inequality~\eqref{eq-desigualdad-adjuntas}. Proceeding as in~\eqref{eq-I+II+III} we get
\begin{equation*}
\int |T^{\star}f|^{p} w \lesssim I^{\star} + II^{\star} + III^{\star},
\end{equation*}
and we can estimate $III^{\star}$ in the same way as $III$, since $T^{\star}$ is also a Calder\'on-Zygmund operator.

For $II^{\star}$, we write
\begin{equation}
\begin{split}
II^{\star} & =\sum_{n\in\NN}\int_{Q_n} |T^{\star}(f\chi_{(2Q_n)^c})|^{p} w
\\& = \sum_{n\in\NN}\int_{Q_n} \left(\int_{(2Q_n)^c}|K(x,y)||f(x)|dx\right)^{p} w(y)dy.
\end{split}
\end{equation}
If $y\in Q_n$ we may use H\"older inequality and condition~\ref{cond-a} to obtain
\begin{equation}
\begin{split}
 \int_{(2Q_n)^c}&|K(x,y)||f(x)|dx
\\ & \leq
\sum_{k\ge 1} \left(
\int_{Q_n^{k+1}\setminus Q_n^k}|K(x,y)|^{p'}dx\right)^{1/p'}
\left( \int_{Q_n^{k+1}} |f|^{p} \right)^{1/p}
\\ & \lesssim  
\sum_{k\ge 1} \left( \int_{Q_n^{k+1}\setminus Q_n^k}
|K(x,y|^sdx
\right)^{1/s}
\left(\int_{Q_n^{k+1}} |f|^{p} \right)^{1/p}
|Q_n^k|^{1/s'-1/p}
\\ & \lesssim
\sum_{k\ge 1} 2^{-kN}\left(\frac{1}{|Q_n^{k+1}|}\int_{Q_n^{k+1}} |f|^{p} \right)^{1/p}
\\ & \lesssim \left[
\sum_{k\ge 1} \frac{2^{-kN}}{|Q_n^{k+1}|}\int_{Q_n^{k+1}} |f|^{p}
\right]^{1/p}\left[
\sum_{k\ge 1} 2^{-kN}
\right]^{1/p'}
\\ & \lesssim\left[
\sum_{k\ge 1} \frac{2^{-kN}}{|Q_n^{k+1}|}\int_{Q_n^{k+1}} |f|^{p}
\right]^{1/p}.
\end{split}
\end{equation}

Therefore,
\begin{equation}
\begin{split}
II^{\star} & \lesssim \sum_{n\in\NN} \sum_{k\in\NN}  2^{-kN} \frac{1}{|Q_n^{k+1}|}\int_{Q_n^{k+1}} |f|^{p} \int_{Q_n^{k+1}} w(y)dy
\\ & \lesssim \sum_{n\in\NN} \sum_{k\in\NN}  2^{-k(N-\theta)} \int_{Q_n^{k+1}} |f|^{p} M^{\theta}w
\\ & \lesssim \sum_{k\in\NN} 2^{-k(N-\theta)}\int_{\RR^d}
\left(\sum_{n\in\NN}\chi_{Q_n^{k+1}}\right) |f|^{p}
M^{\theta}w
\\ & \lesssim \int_{\RR^d}|f|^{p}
M^{\theta}w,
\end{split}
\end{equation}
choosing $N=N_1+\theta+1$.

It only remains to estimate $I^{\star}$. Proceeding as in~\eqref{eq-pasar_a_h}, we have

\begin{equation*}
\begin{split}
I^{\star} & =
\sum_{n\in\NN} \int_{Q_n} |T^{\star}(f\chi_{2Q_n})-T_0^{\star}(f\chi_{2Q_n})|^{p}w
\leq \|h^{\star}\|^{p}_{L^{p}(w)},
\end{split}
\end{equation*}
where
\begin{equation*}
h^{\star}(y)=\int_{B(y,\rho(y))}|Q(x,y)||f(x)|dx,
\end{equation*}
and write
\begin{equation*}
h^{\star}(y)=\sum_{k=0}^{\infty} h_k^{\star}(y)=\sum_{k=0}^{\infty} \int_{B(y,2^{-k}\rho(y))\setminus B(y,2^{-k-1}\rho(y))}|Q(x,y)||f(x)|dx.
\end{equation*}

Now, for a fixed $k$, using H\"older's inequality and denoting $B(y,2^{-k}\rho(y))=B_y^{k}$, we have
\begin{equation}
h_k^\star(y)\leq C \left(\int_{B_y^{k}\setminus B_y^{k-1}}|Q(x,y)|^sdx\right)^{1/s} \left(\int_{B_y^{k}}|f|^{p}\right)^{1/p} (2^{-k}\rho(y))^{d/((s/p')'p') }.
\end{equation}

Now for a fixed $k$, we consider again the covering $\{B(x_{n,k}^l,2\sqrt{d}\gamma 2^{-k}\rho(x_n))\}_{n,l}$. Using condition~\ref{cond-b}, we obtain

\begin{equation}
\begin{split}
\|h_k^\star\|_{L^{p}(w)}^{p} & \leq 
\sum_ {n,l} \int_{B_{n,k}^l} |h_k^\star(y)|^{p}w(y)dy
\\ & \lesssim \sum_{n,l} 
\left(\int_{\be B_{n,k}^l}|f|^{p}\right)
(2^{-k}\rho(y_j))^{dp(1/s'-1/p)}
\int_{B_{n,k}^l}
  \left(\int_{E_{n,k}^l}|Q(x,y)|^sdx\right)^{p/s}  w(y)dy 
\\ & \lesssim \sum_{n,l} \left(\int_{\be B_{n,k}^l}|f|^{p}\right)
\frac{2^{-kp\delta}}{(2^{-k}\rho(y_n))^d}
\int_{\beta{B}_{n,k}^l} w
\\ & \lesssim \sum_{n,l} 2^{-kp\de} \int_{\be B_{n,k}^l}|f(x)|^{p} \left(\frac{1}{|\beta{B}_{n,k}^l|}\int_{\beta{B}_{n,k}^l} w\right)dx
\\ & \lesssim 2^{-kp\delta} \sum_{n,l}\int_{\beta{B}_{n,k}^l}
|f|^{p} M^{\theta}w
\\ & \lesssim 2^{-kp\delta} \|f\|^{p}_{L^{p}(M^{\theta}w)}.
\end{split}
\end{equation}
So, as it was done in~\eqref{eq-de_hk_a_h},
\begin{equation}
\|h^\star\|_{L^{p}(w)}\leq C_\theta \sum_{k}\|h_k^\star\|_{L^{p}(w)} \lesssim \|f\|_{L^p(M^{\theta}w)}.
\end{equation}

Using the estimates obtained for $I^\star$, $II^\star$ and $III^\star$ we arrive to inequality~\eqref{eq-desigualdad-adjuntas}.

\end{proof}

\begin{rem}\label{rem-tipofuerte-p=1}
	It is worth noting that the estimates obtained for $I$ and $II$ also hold for the case $p=1$. Following the same ideas as above we arrive to
	\begin{equation}
	\sum_{n\in\NN}\int_{Q_n} |T(f\chi_{(2Q_n)^c})| w \leq C_\theta\int_{\RR^d}|f| M_{s'}^{\theta}(w),
	\end{equation}
	\begin{equation}
	\sum_{n\in\NN}\int_{Q_n} |T(f\chi_{2Q_n}) - T_0(f\chi_{2Q_n})| w \leq C_\theta \int_{\RR^d}|f| M_{s'}^{\theta}(w).
	\end{equation}
\end{rem}

Now we prove the weak-type inequalities stated in Theorem~\ref{teo-tipo debil}

\begin{proof}[Proof of Theorem~\ref{teo-tipo debil}]
	Let $T$ be a linear operator with kernel $K$ and $w\in L^1_{loc}$, $w\geq 0$. Suppose first that $K$ satisfy conditions~\ref{cond-a} and~\ref{cond-b} for some $1<s<\infty$. Consider again $\{Q_n\}_{n\in\NN}$, the partition of the space associated to $\gamma\rho$, with $\gamma$ chosen as in the proof of Theorem~\ref{teo-prin}. For $\lambda>0$, we may write
	\begin{equation}\label{eq-decomp-tipodebil}
	\begin{split}
	w(\{|Tf|>\lambda\}) & \leq \sum_{n\in\NN} w(\{x\in Q_n: |Tf(x)|>\lambda\})
	\\ & \leq \sum_{n\in\NN} w(\{x\in Q_n: |T(f\chi_{2Q_n})(x)-T_0(f\chi_{2Q_n})(x)|>\lambda/3\})
	\\ & \hspace{0.6cm}+ \sum_{n\in\NN} w(\{x\in Q_n: |T(f\chi_{(2Q_n)^c})(x)|>\lambda/3\})
	\\ & \hspace{1.2cm}+\sum_{n\in\NN} w(\{x\in Q_n: |T_0(f\chi_{2Q_n})(x)|>\lambda/3\})\\
	& = I + II + III.
	\end{split}
	\end{equation}
	To estimate $III$ we can use this time Theorem~\ref{teo-perez-tipodebil} together with Lemma~\ref{lem-gamma0} to get
	\begin{equation}
	\begin{split}
	III &= \sum_{n\in\NN} w(\{x\in Q_n: |T_0(f\chi_{2Q_n})(x)|>\lambda/3\})  
	\\ & \leq \sum_{n\in\NN} w\chi_{Q_n}(\{x: |T_0(f\chi_{2Q_n})(x)|>\lambda/3\})
	\\ & \lesssim \frac{1}{\lambda} \sum_{n\in\NN} \int_{2{Q}_n} |f| M_{A}(w\chi_{Q_n})
	\\ & \lesssim \frac{1}{\lambda} \sum_{n\in\NN} \int_{2{Q}_n} |f| M^{loc}_{A}(w)
	\\ & \lesssim \frac{1}{\lambda}  \int_{\RR^d} |f| M^{loc}_{A}(w),
	\end{split}
	\end{equation}
	for any Young function $A\in\bigcup_{p>1}\mathcal{D}_p$. In particular we can take $A(t)=t^{s'}$ since we will not get any better for the other terms.
	
As for $I$ and $II$ we use the strong type inequalities for $p=1$ stated on Remark~\ref{rem-tipofuerte-p=1}. In this way we obtain \eqref{eq-desigualdad-tipodebil}.	

Now, suppose that the kernel $K$ satisfy conditions~\ref{cond-a_infty} and~\ref{cond-b_infty}. Let $\lambda>0$, we use the same decomposition as in~\eqref{eq-decomp-tipodebil} to get
\begin{equation*}
w(\{|Tf>\lambda\})\leq I + II + III.
\end{equation*}
We deal with $III$ in the same way, obtaining
\begin{equation*}
III\lesssim  \frac{1}{\lambda}\int_{\RR^d}|f| M_{A}^{loc}(w),
\end{equation*}
	for any $A\in\bigcup_{p>1}\mathcal{D}_p$.
	
	For $k\in \mathbb{Z}$ we set $Q_n^k=B(x_j,\gamma2^k\rho(x_j))$. To estimate the term $II$ by the Tchebyshev's inequality we may write 
	\begin{equation}
	\begin{split}
	II & =  \sum_{n\in\NN} w(\{x\in Q_n: |T(f\chi_{(2Q_n)^c})(x)|>\lambda/3\})
	\\& \leq   \sum_{n\in\NN}\frac{3}{\lambda} \int_{Q_n} |T(f\chi_{(2Q_n)^c})|(x)w(x)dx
	\\& \leq   \sum_{n\in\NN}\frac{3}{\lambda}\int_{Q_n} \left(\sum_{k\in\NN}\int_{Q_n^{k+1}\setminus Q_n^k}
	|K(x,y)||f(y)|dy \right) w(x)dx.
	\end{split}
	\end{equation}
Now, using condition~$(a_{\infty})$ 
\begin{equation}
\begin{split}
II &  \lesssim \frac{1}{\lambda} \sum_{n\in\NN}\int_{Q_n}\sum_{k\in\NN}
\frac{2^{-kN}}{(2^k\rho(x_n))^d}\left(\int_{Q_n^{k+1}}
|f(y)|dy \right) w(x)dx 
 \\ & \lesssim \frac{1}{\lambda} \sum_{n\in\NN} \sum_{k\in\NN}
 2^{-kN}  \int_{Q_n^{k+1}}
 |f(y)| \left(\frac{1}{|Q^{k+1}_n|}\int_{Q_n^{k+1}}w(x)dx   \right)dy 
  \\ & \lesssim \frac{1}{\lambda} \sum_{n\in\NN} \sum_{k\in\NN}
  2^{-k(N-\theta)}  \int_{Q_n^{k+1}}
  |f(y)| M^{\theta}(w)dy
  \\ & \lesssim \sum_{k\in\NN} 2^{-k(N-\theta)} \int_{\RR^d}
  \left(\sum_{n\in\NN}\chi_{Q_n^{k+1}}\right)
  |f(y)| M^{\theta}(w)dy
  \\ & \lesssim \int_{\RR^d} |f(y)| M^{\theta}(w)dy,
\end{split}
\end{equation}
choosing $N=N_1+\theta + 1$.

Next, to estimate $I$ we use the Tchebyshev's inequality and condition $(b_{\infty})$.
\begin{equation}
\begin{split}
I & =
\sum_{j\in\NN} w(\{x\in Q_n: |(T-T_0)(f\chi_{2Q_n})(x)|>\lambda/3\})
\\ & \leq
\sum_{n\in\NN} \frac{3}{\lambda}
\int_{Q_n}\left(\int_{2Q_n}|K(x,y)-K_0(x,y)||f(y)|dy\right)w(x)dx
\\ & \lesssim 
\frac{1}{\lambda}\sum_{j\in\NN} 
\int_{Q_n}\left(\int_{2Q_n}\frac{|f(y)|}{|x-y|^d}
\left(\frac{|x-y|}{\rho(x)}\right)^{2-d/q}dy\right)w(x)dx
\\ & \lesssim \frac{1}{\lambda}\sum_{n\in\NN} \rho(x_n)^{d/q-2}
\int_{2Q_n} |f(y)|
\int_{Q_n} |x-y|^{2-d/q-d}w(x)dx\,dy.
\end{split}
\end{equation}
Now, if $y\in 2Q_n$, and calling $B^y_n = B(y,3\ga\rho(x_n))$, then
\begin{equation*}
\begin{split}
\int_{Q_n}& |x-y|^{2-d/q-d} w(x) dx 
 \\ 
& \leq 
\sum_{k\in\NN} \int_{2^{-k+1}B^y_n\setminus2^{-k}B^y_n}
|x-y|^{2-d/q-d} w(x) dx
\\ & \leq 
\rho(x_n)^{2-d/q} \sum_{k\in\NN}
\frac{2^{-k(2-d/q)}}{(2^{-k}\rho(x_n))^d}
\int_{2^{-k+1}B^y_n} w 
\\ & \leq \rho(x_n)^{2-d/q} \Mloc w(y).
\end{split}
\end{equation*}
Therefore, we obtain
\begin{equation}
\begin{split}
I & \lesssim 
\frac{1}{\lambda}\sum_{n} 
\int_{2Q_n} |f|
\Mloc w 
\lesssim \frac{1}{\lambda}
\int_{\RR^d} |f(y)|
\Mloc w.
\end{split}
\end{equation}
Altogether we obtain inequality~\eqref{eq-desigualdad-tipodebilbuena}. The same estimate is obtained for $T^{\star}$ since conditions~\ref{cond-a_infty} and~\ref{cond-b_infty} are symmetric on $x$ and $y$.

\end{proof}

Finally, we end this section with the proof of Theorem~\ref{teo-deigualdades-operad-V}.
\begin{proof}[Proof of Theorem~\ref{teo-deigualdades-operad-V}]
	Let $T$ be a linear operator with associated kernel $K$ satisfying~\ref{cond-c}. First, observe that condition~\ref{cond-c} implies both conditions~\ref{cond-a} and~\ref{cond-b} with $K_0=0$. Then, proceeding as in equation~\eqref{eq-I+II+III} we can write
	\begin{equation}
	\begin{split}
	\int |Tf|^p w
	&  = \sum_{n\in\NN} \int_{Q_n} |T(f\chi_{2Q_n}) + T(f\chi_{2Q_n^c})|^p w
	\\& \lesssim \sum_{n\in\NN} \int_{Q_n} |T(f\chi_{2Q_n})|^p w
	+ \sum_{n\in\NN} \int_{Q_n} |T(f\chi_{2Q_n^c})|^p w\\
	& =	I + II.
	\end{split}
	\end{equation}
	Then, inequality~\eqref{eq-desigualdad-derechas} holds for $1\leq p<s$ following the same lines as in the proof of Theorem~\ref{teo-prin} and taking into account Remark~\ref{rem-tipofuerte-p=1} for $p=1$.
	
	To obtain estimate~\eqref{eq-desigualdad-adjuntas_conV} we proceed as above to get
	\begin{equation}
	\begin{split}
	\int |T^{\star}f|^p w
 \leq \sum_{n\in\NN} \int_{Q_n} |T^{\star}(f\chi_{2Q_n})|^p w
	+ \sum_{n\in\NN} \int_{Q_n} |T^{\star}(f\chi_{2Q_n^c})|^p w
	=
	I^{\star} + II^{\star}
	\end{split}
	\end{equation}
	and we deal with $I^{\star}$ and $II^{\star}$ as in the proof of Theorem~\ref{teo-prin}.
	
\end{proof}

\section{Application to Schr\"odinger operators}\label{sec-aplSchrodinger}
In this section we apply our general results to operators associated to the semi-group generated by the Schr\"odinger differential operator $L=-\Delta + V$ on $\RR^d$ with $d\geq 3$. We will always suppose that the potential $V$ is a non-negative function, non-identically zero, satisfying a reverse H\"older condition of order $q>d/2$. 
Under these assumptions the function $\rho$ defined by
\begin{equation}
\rho(x)=\sup \left\{r>0:\frac{1}{r^{d-2}}\int_{B(x,r)}V\leq 1\right\}, \,\,x\in\RR^d
\end{equation}
is a critical radius function, that is, property~\eqref{eq-constantesRho} is satisfied for some constants $C_0$ and $N_0$.

It is known that $V\in RH_q$, $q>1$ implies that $V$ is a doubling measure, i.e.\ there exists $C_1$ such that
\begin{equation}
\int_{B(x,2r)}V\leq C_1 \int_{B(x,r)}V.
\end{equation}
In fact, if $V\in RH_q$, $q>1$, then $V$ belongs to the $A_{\infty}$ class of Muckenhoupt. 

The following is an useful inequality for $V\in RH_q$ with $q>d/2$ that follows easily from Lemma~1.2 and Lemma~1.8 in~\cite{shen}.

\begin{lem}\label{lem-V}
	Let $V\in RH_q$ for some $q>d/2$. Let $N_1=\log_2C_1+2-d$, where $C_1$ is the doubling constant of $V$. Then, for any $x_0\in\RR^d$, $R>0$,
	\begin{equation*}
	\frac{1}{R^{d-2}}\int_{B(x_0,R)}V(y)dy\leq C\left(1+\frac{R}{\rho(x_0)}\right)^{N_0}\left(1+\frac{\rho(x_0)}{R}\right)^{d/q-2}.
	\end{equation*}
\end{lem}
\subsection{Riesz-Schr\"odinger transforms}
We consider the operators $\mathcal{R}_1=\nabla L^{-1/2}$ and $\mathcal{R}_2=\nabla^2L^{-1}$, the Riesz-Schr\"odinger transforms of order~$1$ and~$2$ respectively. Let $K_1$ and $K_2$ be their associated kernels. 

The size condition~\ref{cond-a} was shown to hold in~\cite{BHQ2018}, for both $K_1$ and $K_2$. To prove that these kernels satisfy also condition~\ref{cond-b} we will compare them with the classical Riesz transforms $R_1=\nabla(-\Delta)^{-1/2}$ and $R_2=\nabla^2(-\Delta)^{-1}$ and their associated kernels $K_{0,1}$ and $K_{0,2}$.

Before proving condition~\ref{cond-b} we state the following lemmas that provide us estimates for the difference between the kernels associated to the Riesz-Schr\"odinger transforms and the classical ones. For the Riesz-Schr\"odinger transform of order 1 such result was already obtained by Shen. On the other hand, the estimate corresponding to the second order operator is new and we believe is interesting in its own right.
\begin{lem}\label{lem-compR1}[See~\cite{shen}, inequality~$(5.9)$]
	Let $V\in RH_q$ for $d/2<q<d$. There exists $C$ such that
	\begin{equation*}
	|K_1(x,y)-K_{0,1}(x,y)|\leq 
	\frac{C}{|x-y|^{d-1}}\left(
	G(x,y)+\frac{1}{|x-y|}\left(
	\frac{|x-y|}{\rho(x)}\right)^{2-d/q}\right),
	\end{equation*}
	where
	\begin{equation}\label{eq-defG}
	G(x,y)=\int_{B(x,|x-y|/4)}\frac{V(u)}{|u-x|^{d-1}}du.
	\end{equation}
\end{lem}

\begin{lem}\label{lem-comparacionR2}
	Let $x$, $y_0\in\RR^d$ and $R>0$ such that $R\leq|y-x_0|\leq \rho(x_0)$. Let $x\in B(x_0,R/8)$. Then there exists a constant $C$ such that
	\begin{equation*}
	|K_2(x,y)-K_{0,2}(x,y)|\leq
	 C |R_2(V\Gamma(y,\cdot)\chi_{B(x_0,R/4)})(x)|
	+\frac{C}{R^d}\left(\frac{R}{\rho(x_0)}\right)^{\delta},
	\end{equation*}
	with $\delta=\min\{1,2-d/q\}$.
\end{lem}

\begin{proof}
	Let $\Gamma$ and $\Gamma_0$ be the fundamental solution of $L$ and $-\Delta$ respectively. As it was shown in~\cite{shen}, page $540$,
	\begin{equation}
	\Gamma(x,y)- \Gamma_0(x,y) = -\int_{\RR^d}\Gamma_0(x,\xi)V(\xi)\Gamma(y,\xi) d\xi.
	\end{equation} 
	From this we get the following expression for the difference of the kernels.
	\begin{equation}
	K_2(x,y)- K_{0,2}(x,y) = \nabla_1^2\Gamma(x,y)- \nabla_1^2\Gamma_0(x,y) = -\nabla_1^2\int_{\RR^d}\Gamma_0(x,\xi)V(\xi)\Gamma(y,\xi) d\xi.
	\end{equation} 
	Next, we define the following domains: $J_1=B(x_0,R/4)$, $J_2=B(y,R/4)$ and $J_3=(J_1\cup J_2)^c$. The term corresponding to the integral over $J_1$ is, upon a sign, the classical second order Riesz transform applied to function in $L^q$ with compact support, that is
	\begin{equation}
	|\nabla_1^2\int_{J_1}\Gamma_0(x,\xi)V(\xi)\Gamma(y,\xi)d\xi| = |R_2(V\Gamma(y,\cdot)\chi_{B(x_0,R/4)})(x)|.
	\end{equation}
	
	On $J_2$, since we are away from the singularity of $\Gamma_0$, we can use the size estimates for $\Gamma$ and $\Gamma_0$ together with H\"older's inequality to obtain
	\begin{equation}
	\begin{split}
	\left|\int_{J_2}\nabla_1^2\right.&\Gamma_0(x,\xi)V(\xi)\Gamma(y,\xi)d\xi\Bigg|
	\\ & \leq
	\frac{C}{R^d} \int_{B(y,R/4)} \frac{V(\xi)}{|y-\xi|^{d-2}}d\xi
	\\ & \leq
	\frac{C}{R^d} \left(\int_{B(y,R/4)}V^q(\xi)d\xi\right)^{\tfrac{1}{q}}
	\left(\int_{B(y,R/4)}\frac{d\xi}{|y-\xi|^{(d-2)q'}}\right)^{\tfrac{1}{q'}}.
	\end{split}
	\end{equation}
	For the first integral we can use the reverse H\"older condition for $V$ together with Lemma~\ref{lem-V}, while on the second integral $q>d/2$ implies that $(d-2)q'<d$. Then
	\begin{equation}
	\left|\int_{J_2}\nabla_1^2\Gamma_0(x,\xi)V(\xi)\Gamma(y,\xi)d\xi\right| \lesssim
	\frac{1}{R^d}\left(\frac{R}{\rho(x_0)}\right)^{2-d/q},
	\end{equation}
	since $y\in B(x_0,\rho(x_0))$.
	
	To estimate the integral on $J_3$ we divide in $J_{31}\cup J_{32}$, where
	$J_{31}=\{\xi \in \RR^d: R/4\leq|y-\xi|<2R \,\land\, |x_0-\xi|\geq R/4\}$ and  $J_{32}=\{\xi \in \RR^d:|y-\xi|\geq 2R\}.$
	On $J_{31}$ we are away from the singularities of both $\Gamma$ y $\Gamma_0$, then
	\begin{equation}
	\begin{split}
	\left|\int_{J_{31}}\nabla_1^2\Gamma_0(x,\xi)V(\xi)\Gamma(y,\xi)d\xi\right|
	& \lesssim
	\int_{J_{31}} \frac{V(\xi)}{|x-\xi|^d |y-\xi|^{d-2}} d\xi \\ 
	& \lesssim
	 \frac{1}{R^{2d-2}} \int_{B(y,2R)} V(\xi)d\xi \\
	 & \lesssim
	\frac{1}{R^d} \left(\frac{R}{\rho(x_0)}\right)^{2-d/q},
	\end{split}
	\end{equation}
	where we in the last inequality, have used again Lemma~\ref{lem-V}.
	
	Regarding $J_{32}$ it is easy to check that $|x-\xi|\geq 3|y-\xi|/8$, so 
	\begin{equation}
	\begin{split}
	\left|\int_{J_{32}}\nabla_1^2\Gamma_0(x,\xi)V(\xi)\Gamma(y,\xi)d\xi\right|
	& \leq C_N
	\int_{J_{32}} \frac{V(\xi)}{|x-\xi|^d |y-\xi|^{d-2}} \left(1+\frac{|y-\xi|}{\rho(y)}\right)^{-N} d\xi  
	\\ & \leq C_N
	\int_{J_{32}} \frac{V(\xi)}{ |y-\xi|^{2d-2}}\left(1+\frac{|y-\xi|}{\rho(y)}\right)^{-N} d\xi.
	\end{split}
	\end{equation}
	
	Assume firts that $2R<\rho(y)$. We split the integral in $J_{321}=\{\xi\in\RR^d: 2R\leq |y-\xi| <\rho(y)\}$ and
	$J_{322}=\{\xi\in\RR^d: |y-\xi|\geq \rho(y)\}$. For the integral on $J_{321}$, let $k_0\in\NN$ such that  $2^{k_0-1}R\leq\rho(y)\leq2^{k_0}R $. Then using Lemma~\ref{lem-V} and that $d>2-d/q$,
	\begin{equation}
	\begin{split}
	\int_{2R\leq |y-\xi| <\rho(y)} \frac{V(\xi)}{ |y-\xi|^{2d-2}}d\xi 
	& \leq
	\sum_{k=2}^{k_0}\int_{2^{k-1}R\leq|y-\xi|< 2^kR} \frac{V(\xi)}{|y-\xi|^{2d-2}}
	\\ & \lesssim
	\sum_{k=1}^{k_0} \frac{1}{(2^kR)^d} \frac{1}{(2^kR)^{d-2}} \int_{B(y,2^kR)} V(\xi)d\xi 
	\\ & \lesssim 
	\frac{1}{R^d} \sum_{k=1}^{k_0} 2^{-kd} \left(\frac{2^kR}{\rho(y)}\right)^{2-d/q}
	\\ & \lesssim
	\frac{1}{R^d} \left(\frac{R}{\rho(x_0)}\right)^{2-d/q},
	\end{split}
	\end{equation}
	since $\rho(y)\simeq\rho(x_0)$.
	
	On $J_{322}$, let $\mu=\log_2 C_1$, where $C_1$ is the doubling constant of the potential $V$. Then we have
	\begin{equation}
	\begin{split}
	\int_{|x-\xi| \geq\rho(y)} &\frac{V(\xi)}{ |y-\xi|^{2d-2}}\left(\frac{\rho(y)}{|y-\xi|}\right)^N d\xi 
	\\& \lesssim
	\sum_{k=1}^{\infty}\frac{1}{2^{kN}}\int_{2^{k-1}\rho(y)|y-\xi|< 2^{k+}\rho(y)} \frac{V(\xi)}{|y-\xi|^{2d-2}}
	\\ & \leq
	\sum_{k=1}^{\infty}\frac{1}{2^{k(2d-2+N)}\rho(y)^{2d-2}}\int_{B(y,2^k\rho(y))} V(\xi)d\xi
	\\ & \lesssim 
	\frac{1}{\rho(y)^d}\sum_{k=1}^{\infty} \frac{1}{2^{k(2d-2+N-\mu)} \rho(y)^{d-2}}\int_{B(y,\rho(y))}V(\xi)d\xi
	\\ & \lesssim 
	\frac{1}{\rho(y)^d} \leq
	\frac{1}{R^d} \left(\frac{R}{\rho(x_0)}\right)^{2-d/q},
	\end{split}
	\end{equation}
	choosing $N$ big enough and using that $\rho(y)\simeq\rho(x_0)$, $R<\rho(x_0)$ and $2-d/q<d$.
	
\end{proof}

As an applycation of Lemma~\ref{lem-comparacionR2}, Theorem~\ref{teo-prin} and Theorem~\ref{teo-tipo debil} we obtain the following inequalities for $\mathcal{R}_2$.

\begin{thm}\label{teo-R1}
	Let $V\in RH_q$ for $q>d/2$, and $\theta\geq 0$. Then, for any weight $w$ the following inequalities hold.
	\begin{equation}
	\int |\mathcal{R}_2 f|^p w \leq C_{\theta}\int |f|^p M_r^{\theta}w,
	\end{equation}
	for $1<p<q$ and $r=(q/p)'$,
	\begin{equation}
	\int |\mathcal{R}_2^{\star}f|^{p} w \leq C_{\theta,A}\int |f|^{p} (M_A^{loc}+ M^{\theta})w,
	\end{equation}
	for $q'<p<\infty$ and any Young function $A\in\mathcal{D}_p$,
	\begin{equation}
	w(\{|\mathcal{R}_2 f|>\lambda\})\leq 
	\frac{C_\theta}{\lambda}\int |f| M_{q'}^{\theta}w.
	\end{equation}
\end{thm}

\begin{proof}
	As we was say before, it only remains to check condition~\ref{cond-b'} for the kernel $K_2$. Let $x_0$, $y\in\RR^d$ and $R>0$ such that $R<|y-x_0|<2R$ and $R\leq \rho(x_0)$. We are going to check condition~\ref{cond-b'} with $s=q$ using Lemma~\ref{lem-comparacionR2},
	\begin{equation*}
	\begin{split}
	\Bigg(\int_{B(x_0,R/2)}&|K_2(x,y)-K_{2,0}(x,y)|^q dx\Bigg)^{1/q}
	\\& \leq  	\left(\int_{B(x_0,R/2)}\left(|R_2(V\Gamma(y,\cdot)\chi_{B(x_0,R/4)})(x)|
	 + \frac{C}{R^d}\left(\frac{R}{\rho(x_0)}\right)^{\delta}\right)^q dx\right)^{1/q}.
	\end{split}
	\end{equation*}  

Dividing the integral in two terms it is straightforward that the second one gives us the desired estimate. For the first one, recalling that $R_2$ is a bounded operator on $L^q$ for $1<q<\infty$, and applying Lemma~\ref{lem-V},
\begin{equation}
\begin{split}
\Bigg(\int_{B(x_0,R/2)}|R_2(V\Gamma(y,\cdot)\chi_{B(x_0,R/4)})&(x)|^q
dx\Bigg)^{1/q}
\\ & \lesssim
\Bigg(\int_{B(x_0,R/4)}V^q(x)|\Gamma(y,x)|^q 
dx\Bigg)^{1/q}
\\ & \lesssim
\frac{1}{R^{d-2}}\left(\int_{B(x_0,R/4)} V^q\right)^{1/q}
\\ & \lesssim
R^{-d/q'} \left(\frac{R}{\rho(x_0)}\right)^{2-d/q}.
\end{split}
\end{equation}

\end{proof}

For $\mathcal{R}_1$ different inequalities hold true depending on $q$. For $q>d$, Shen showed in~\cite{shen} that $\mathcal{R}_1$ and $\mathcal{R}_1^{\star}$ are Calder\'on-Zygmund operators. Moreover, their associated kernels satisfy the stronger size condition~\eqref{cond-a_infty} (see inequality (6.5) there). Later on, condition \eqref{cond-b_infty} was proved for the difference between $K_1$ and $K_{1,0}$ (see~\cite{BHS-RieszJMAA} Lemma~$3$).

Therefore, as an application of Theorem~\ref{teo-prin}, Corollary~\ref{cor-infty} and Theorem~\ref{teo-tipo debil} we obtain the following result.

\begin{thm} 
	Let $V\in RH_q$ for $q>d/2$, $w\geq 0$, $w\in L^1_{loc}$ and $\theta\geq 0$. Let $p_0$ such that $1/p_0=(1/q-1/d)^+$ Then, for $1<p<p_0$ the following inequalities hold.
	\begin{equation}\label{eq-R1aux1}
	\int |R_1 f|^p w \leq C_{\theta}\int |f|^p M_r^{\theta}w,
	\end{equation}
	for $1<p<p_0$ and $r=(p_0/p)'$,
	\begin{equation}\label{eq-R1aux2}
	\int |\mathcal{R}_1^{\star}f|^{p} w \leq C_{\theta}\int |f|^{p} (M_A^{loc}w + M^{\theta}w),
	\end{equation}
	for $p_0'<p<\infty$ and any Young function $A\in \mathcal{D}_p$,
	\begin{equation}\label{eq-R1aux3}
	w(\{|\mathcal{R}_1 f|>\lambda\})\leq 
	\frac{C_\theta}{\lambda}\int |f| M_{p_0'}^{\theta}w.
	\end{equation}
	Moreover, if $q>d$, we have
	\begin{equation}\label{eq-R1aux4}
	\int |\mathcal{R}_1 f|^p w \leq C_{\theta}\int |f|^p (M_A^{loc}w + M^{\theta}w),
	\end{equation}
	for $1<p<\infty$ and any Young function $A\in\mathcal{D}_p$ and
	\begin{equation}\label{eq-R1aux5}
	w(\{|\mathcal{R}_1 f|>\lambda\})\leq 
	\frac{C_\theta}{\lambda}\int |f| M_{A}^{\theta}w,
	\end{equation}
	\begin{equation}\label{eq-R1aux6}
	w(\{|\mathcal{R}_1^{\star} f|>\lambda\})\leq 
	\frac{C_\theta}{\lambda}\int |f| M_{A}^{\theta}w,
	\end{equation}
	for any Young function $A\in\bigcup_{p>1}\mathcal{D}_p$.
\end{thm}

\begin{proof}
Let $V\in RH_q$ for $q>d/2$. Suppose first that $q<d$ and let us show  that $K_1$ satisfy conditions~\ref{cond-b'} and~\ref{cond-a'} for $s=p_0$ with $1/p_0=1/q-1/d$. For~\ref{cond-b'} let $x_0\in\RR^d$, $0<R\leq\rho(x_0)$ and $R<|y-x_0|<2R$. First, we make use of Lemma~\ref{lem-compR1}. Due to the boundedness of the classical fractional integral operator $I_1$ and the reverse H\"older property of $V$ we get that, for $G$ defined in~\eqref{eq-defG},
\begin{equation}\label{eq:compHormA(x,y)}
\begin{split}
\left(\int_{B(x_0,R/2)}\!\!\left(\frac{G(x,y)}{|x-y|^{d-1}}\right)^{p_0} \!\!\!\!dx\right)^{1/p_0}\!\!\!\!\!\!\!\! & \leq \frac{C}{R^{d-1}}\left(\int_{B(x_0,R/2)}\!\!\left(\int_{B(x_0,R)}\frac{V(u)}{|u-x|}du\right)^{p_0}\!\!\!\!\!\!dx\right)^{1/p_0} 
\\ & \leq 
\frac{C}{R^{d-1}} \left(\int_{\RR^d} |I_1(\chi_{B(x_0,R)}V)|^{p_0}\right)^{1/p_0}
\\ & \leq 
\frac{C}{R^{d-1}} \left(\int_{B(x_0.R)} V^{q}\right)^{1/q}
\\ & \leq 
C\frac{R^{d/q-d}}{R^{d-1}}  \int_{B(x_0,R)} V
\\ & \leq 
C R^{-d/p_0'}\left(\frac{R}{\rho(x_0)}\right)^{2-d/q},
\end{split}
\end{equation}
where, in the last inequality, we have used Lemma~\ref{lem-V}. As for the second term appearing in Lemma~\ref{lem-compR1}, the same estimate holds easily. To check the size condition~\ref{cond-a'} we make use of the following estimate that can be found on page 538 of~\cite{shen}. For every $N>0$ there exists a constant $C_N$ such that
\begin{equation}
|K_1(x,y)|\leq C_N\left(1+\frac{|x-y|}{\rho(x)}\right)^{-N}
\frac{1}{|x-y|^{d-1}}\left(
G(x,y) + \frac{1}{|x-y|}\right),
\end{equation}
with $G$ as above. This estimate, together with a similar argument as in~\eqref{eq:compHormA(x,y)} gives us \ref{cond-a'} for $s=p_0$. Therefore, inequalities~\eqref{eq-R1aux1},~\eqref{eq-R1aux2} and~\eqref{eq-R1aux3} follow as an application of Theorems~\ref{teo-prin} and~\ref{teo-tipo debil}.

Next, suppose that $q>d$. In this case, it is known that $K_1$ satisfy the point-wise estimates~\ref{cond-a_infty} and~\ref{cond-b_infty}. For the size condition we refer to inequality~$(6.5)$ in~\cite{shen}. Condition~\ref{cond-b_infty} was stated and proved in~\cite{BHS-RieszJMAA}, Lemma~3. Thus, applying now Corollary~\ref{cor-infty} and Theorem~\ref{teo-tipo debil} we obtain inequalities~\eqref{eq-R1aux4},~\eqref{eq-R1aux5} and~\eqref{eq-R1aux6}. 

\end{proof}

\subsection{Operators $V^{\gamma}L^{-\gamma}$}
We consider, for $V\in RH_q$, $q>d/2$, the family of operators of type $V^{\gamma}L^{-\gamma}$ for $0<\gamma<d/2$. For each $\gamma$, we can write $K_\gamma$, the kernel of $V^{\gamma}L^{-\gamma}$, as
\begin{equation*}
K_\gamma(x,y)=V^\gamma(x) J_{\gamma}(x,y),
\end{equation*}
where $J_\gamma$ is the corresponding kernel of the fractional integral operator $L^{-\gamma}$. For $J_\ga$ we have the following estimate that can be found in~\cite{MR3180934_stinga}, page $587$. For each $N>0$ there exists $C_N$ such that
\begin{equation}\label{eq:tam_Jgamma}
|J_{\gamma}(x,y)|\leq \frac{1}{|x-y|^{d-2\gamma}} 
{C_N}{\left(1+\frac{|x-y|}{\rho(x)}\right)^{-N}}.
\end{equation}
 We will show next that the size estimate for $J_\gamma$ gives us condition~\ref{cond-c} for $K_\gamma$ with $s=q/\gamma$. In fract, let $x_0$, $y\in\RR^d$ and $R$ such that $|y-x_0|<R/2$. Applying Lemmas~\ref{lem-trucho} and~\ref{lem-V} we get
\begin{equation}
\begin{split}
\Bigg(\int_{R<|x-x_0|<2R}&|K_\gamma(x,y)|^{q/\gamma}dx\Bigg)^{\gamma/q}
 \\ & \leq \frac{C_N}{R^{d-2\gamma}} \left(1+\frac{R}{\rho(x_0)}\right)^{-N/N_0}
 \left(\int_{B(x_0,2R)}V^q\right)^{\gamma/q}
\\ & \lesssim R^{-d/(q/\gamma)'}\left(1+\frac{R}{\rho(x_0)}\right)^{-N/N_0+\ga N1}
\left(1+\frac{\rho(x_0)}{R}\right)^{-\gamma(2-d/q)}.
\end{split}
\end{equation}

The above estimate together with Theorem~\ref{teo-deigualdades-operad-V} give us the following result.
\begin{thm}\label{teo:Vcongamma}
	Let $V\in RH_q$ for $q>d/2$, $0<\gamma<d/2$ and $\theta\geq 0$. Then, for any weight $w$,
	\begin{equation}
	\int |V^{\gamma}L^{-\gamma} f|^p w \leq C_{\theta}\int |f|^p M_r^{\theta}w,
	\end{equation}
	for $1\leq p < q/\gamma$, $r=(q/(\gamma p))'$ and
	\begin{equation}
	\int |L^{-\gamma} V^{\gamma}f|^{p} w \leq C_{\theta}\int |f|^{p} M^{\theta}w,
	\end{equation}
	for $(q/\gamma)'<p<\infty$.
	\end{thm}

\subsection{Operators $V^{\gamma-1/2}\nabla L^{-\gamma}$}
We consider the family of operators $V^{\gamma-1/2}\nabla L^{-\gamma}$ for $1/2<\gamma\leq 1$ that includes the operator $L^{-1}\nabla V^{1/2}$ which appeared first in~\cite{shen}. In~(cita) it was shown that the associated kernel $K^{\gamma}$ can be written as the product $K^{\gamma}(x,y)=V^{\nu/2}(x)K_{\nu}(x,y)$, with $K_{\nu}$ a fractional kernel of order $\nu=2\gamma-1$, satisfying for each $N$,
\begin{equation}\label{tamKnubueno}
|K_{\nu}(x,y)|\leq \frac{C_N}{|x-y|^{d-2\gamma+1}}\left(1+\frac{|x-y|}{\rho(y)}\right)^{-N},
\end{equation}
if $V\in RH_q$ with $q>d$ and 
\begin{equation}\label{tamKnumalo}
\begin{split}
\left(\int_{R<|x-y|<2R}|K_\nu(x,y)|^{p_0} dx\right)^{1/p_0} & 
\leq {C} R^{-d/p_0'+2\gamma-1}\left(1+\frac{R}{\rho(y)}\right)^{-N},
\end{split}
\end{equation}
when $d/2< q < d$, with $p_0$ such that $1/p_0=1/q-1/d$.

 We will show now that these estimates imply condition~\ref{cond-c} for $s=p_\gamma$ such that
\begin{equation}\label{eq-def-s-Knu}
\frac{1}{p_\gamma}=\left(\frac{1}{q}-\frac{1}{d}\right)^+ + \frac{2\gamma-1}{2q}.
\end{equation}

Let $x_0$, $y\in\RR^d$ and $R>0$ such that $|y-x_0|<R/2$. If $q>d$, and hence $\frac{1}{p_\gamma}=\frac{2\gamma-1}{2q}$, we may use estimate~\eqref{tamKnubueno}, condition \eqref{eq-RH} and Lemma~\ref{lem-V} to get
\begin{equation}
\begin{split}
\Bigg(\int_{R<|x-x_0|<2R}&|K_\gamma(x,y)V^{\gamma-1/2}(x)|^{\frac{2q}{2\gamma-1}}dx\Bigg)^{\frac{2\gamma-1}{2q}}
\\& \leq \frac{C_N}{R^{d-2\gamma+1}}\left(\int_{B(x_0,2R)}V^q\right)^{\frac{2\gamma-1}{2q}}
\left(1+\frac{R}{\rho(x_0)}\right)^{-N}
\\ & \lesssim R^{-d/p_\gamma'}\left(1+\frac{\rho(x_0)}{R}\right)^{-(\gamma-1/2)(2-d/q)}\left(1+\frac{R}{\rho(x_0)}\right)^{-N+N_1(\ga-1/2)}.
\end{split}
\end{equation}

If $d/2<q<d$, now we have $\frac{1}{p_\gamma}=\frac{1}{p_0} + \frac{2\gamma-1}{2q}$. Then, by Holder's inequality together with~\eqref{tamKnumalo} and Lemma~\ref{lem-V} as above we obtain
\begin{equation}
\begin{split}
\Bigg(\int_{R<|x-x_0|<2R}&|K_\gamma(x,y)V^{\gamma-1/2}(x)|^{p_\gamma}dx\Bigg)^{1/p_\gamma}
\\ &\leq  \Bigg(\int_{R<|x-x_0|<2R}|K_\gamma(x,y)|^{p_0}dx\Bigg)^{1/p_0}
\left(\int_{B(x_0,2R)}V^q\right)^{\frac{2\gamma-1}{2q}}
\\ & \lesssim R^{-d/p_\gamma'}\left(1+\frac{\rho(x_0)}{R}\right)^{-(\gamma-1/2)(2-d/q)}\left(1+\frac{R}{\rho(x_0)}\right)^{-N+N_1(\ga-1/2)}.
\end{split}
\end{equation}

Applying the above estimates and Theorem~\ref{teo-deigualdades-operad-V} we obtain the following result.
\begin{thm}\label{teo:mixtos}
	Let $V\in RH_q$ for $q>d/2$, $1/2<\gamma\leq1$, and $\theta\geq 0$. 
	Then if $p_\gamma$ is given by~\eqref{eq-def-s-Knu}, for any weight $w$ we have
	\begin{equation}
	\int |V^{\gamma-1/2}\nabla L^{-\gamma} f|^p w \leq C_{\theta}\int |f|^p M_r^{\theta}w,
	\end{equation}
	for $1\leq p<p_\gamma$ with $r=(p_\gamma/p)'$, and
	\begin{equation}
	\int |L^{-\gamma}\nabla V^{\gamma-1/2}f|^{p} w \leq C_{\theta}\int |f|^{p} M^{\theta}w,
	\end{equation}
	for $p_\gamma'<p<\infty$.
\end{thm}

\section{On local integrability of $Tf$ and $T^{\star}f$}\label{sec-chiB}
In this section we are going to apply the general results of Section~\ref{sec-general} to weights of the form $w=\chi_{B}$. Studying maximal operators like $M^\theta_\phi$ acting on such weights we are going to get sufficient conditions on $f$ to assume some local integrability of $Tf$. We do that in the next lemma. 

\begin{lem} \label{lem-acotacion-maximal-caracteristicaB}
	Let $\theta\geq 0$, $\phi$ a Young function and $Q=B(x_0,\rho(x_0))$ then there exist positive constants $c_1$, $c_2$, $\sig_1$ and $\sig_2$ such that
	
	\begin{equation}\label{max_phi_up_down}
	c_1\left(1+\frac{|x-x_0|}{\rho(x_0)}\right)^{-\sig_1}\leq M_{\phi}^{\theta} \chi_Q(x)
	\leq c_2 \left(1+\frac{|x-x_0|}{\rho(x_0)}\right)^{-\sig_2}
	\end{equation}
\end{lem}
\begin{proof}
	Let $Q=B(x_0,\rho(x_0))$ be a critical ball, $\theta\geq 0$ and $\phi$ a Young function. We may suppose without loss of generality that $\phi(1)=1$. Recalling that 
$$M_{\phi}^{\theta} \chi_Q(x)=\sup_{B(x_B,r_B)\ni x}
	\left(1+\frac{r_B}{\rho(x_B)}\right)^{-\theta} \|\chi_Q\|_{\phi,B},$$
	it is enough to consider $B$ such that $Q\cap B\neq \emptyset$, otherwise $\|\chi_Q\|_{\phi,B}=0$, since
	\begin{equation*}
	\begin{split}
	\|\chi_Q\|_{\phi,B} & = \inf\left\{\lambda: \frac{1}{|B|}\int_B \phi\left(\frac{\chi_Q}{\lambda}\right)\leq 1\right\}
	\\ & =  \inf\left\{\lambda: \frac{1}{|B|}\int_{B\cap Q} \phi\left(\frac{1}{\lambda}\right)\leq 1\right\}.
	\end{split}
	\end{equation*}
	
	Let us consider first a ball $B=B(x_B,r_B)$ with $r_B\leq \rho(x_B)$, and $x\in B$. Then for $y\in B\cap Q$,
	$$|x-x_0|\leq |x-y|+|y-x_0|\leq 2r_B + \rho(x_0)\leq 2\rho(x_B)+\rho(x_0).$$
	Also, since $B$ is sub-critical, $Q$ is critical and  $B\cap Q\neq \emptyset$ we have that $\rho(x_B)\simeq \rho(y)\simeq \rho(x_0)$. Then,
	\begin{equation}
	|x-x_0|\leq  \tilde{C}\rho(x_0),
	\end{equation} 
	for some $\tilde{C}>0$. Then if $x\notin \tilde{Q}=B(x_0,\tilde{C}\rho(x_0))$ we have
	$$\Mloc_{\phi}(\chi_Q)(x)=\sup_{\substack{B\ni x\\r_B\leq\rho(x_B)}} \|\chi_Q\|_{\phi,B}=0.$$
	
	If $x\in \tilde{Q}$ and $B\cap Q\neq \emptyset$,
	\begin{equation}
	\begin{split}
	\|\chi_Q\|_{\phi,B} & = \inf\left\{\lambda: \frac{|B\cap Q|}{|B|} \phi\left(\frac{1}{\lambda}\right)\leq 1\right\}
	\\ & \leq \inf\left\{\lambda: \phi\left(\frac{1}{\lambda}\right)\leq 1\right\}
	\\ & = 1/\phi^{-1}(1)=1.
	\end{split}
	\end{equation}
	So, taking the supreme over all balls we have that if $x\in \tilde{Q}$,
	\begin{equation}
	\Mloc_{\phi}(\chi_Q)(x)\leq \left(1+\frac{r_B}{\rho(x_B)}\right)^{-\sig},
	\end{equation} 
	for any $\sig>0$.
	
	Next, we consider the operator 
	\begin{equation}
	\begin{split}
	M^{\theta, \text{glob}}_{\phi}(\chi_Q)(x)&=\sup_{\substack{B\ni x\\r_B\geq \rho(x_B)}} \left(1+\frac{r_B}{\rho(x_B)}\right)^{-\theta} \|\chi_Q\|_{\phi,B}
	\end{split}
	\end{equation}
	
	As above, it is enough to consider balls $B$ such that $Q\cap B\neq \emptyset$.
	Let $y\in Q \cap B$, then $\rho(y)\simeq\rho(x_0)$.  Using Lemma~\ref{lem-trucho}
	$$\left(1+ \frac{r_B}{\rho(x_B)}\right)^{-\theta}
	\leq C\left(1+ \frac{r_B}{\rho(y)}\right)^{-\theta/N_0}
	\leq C\left(1+ \frac{r_B}{\rho(x_0)}\right)^{-\theta/N_0} $$
	
	Let $x\in B$. Suppose first that $x\notin 2Q$, then 
	$$|x-x_0|\leq |x-y|+|y-x_0|\leq 2r_B + \rho(x_0) \leq 2r_B + |x-x_0|/2$$
and hence $|x-x_0|\leq 4r_B$. Therefore,
	$$\left(1+ \frac{r_B}{\rho(x_B)}\right)^{-\theta}
	\leq C\left(1+ \frac{|x-x_0|}{\rho(x_0)}\right)^{-\theta/N_0}.$$
	
	As before, we have $\|\chi_Q\|_{\phi, B}\leq 1$. 
	Then, if $x\notin 2{Q}$
	$$M^{\theta, \text{glob}}_{\phi}(\chi_Q)(x)\leq C \left(1+\frac{|x-x_0|}{\rho(x_0)}\right)^{-\sigma},$$
	where $\sigma=\theta/N_0$.
	
	On the other hand, if $x\in 2{Q}$,
	$$M^{\theta, \text{glob}}_{\phi}(\chi_Q)(x)\leq M_{\phi} (\chi_Q)(x)\leq 1.$$
	Then, since $|x-x_0|/\rho(x_0)\leq 2$
	$$M^{\theta, \text{glob}}_{\phi}(\chi_Q)(x)\leq C \left(1+\frac{|x-x_0|}{\rho(x_0)}\right)^{-\sigma}.$$
	
	Using that $M_{\phi}^{\theta} \leq \Mloc_{\phi} + M^{\theta, \text{glob}}$ and collecting last estimates we obtain the right hand side of \eqref{max_phi_up_down}. For the boundedness by below, given $x$ we consider $B_x=B(x,|x-x_0|+\rho(x_0))$. Then $x\in B_x$ and $\|\chi_Q\|_{\phi,B_x}=1$. Therefore,
	
	\begin{equation*}
	\begin{split}
	M_{\phi}^{\theta}(x) & \ge \left(1+\frac{|x-x_0|+\rho(x_0)}{\rho(x_0)}\right)^{-\theta}\|\chi_Q\|_{\phi,B_x}\\
	& \ge 2^\theta\left(1+\frac{|x-x_0|}{\rho(x_0)}\right)^{-\theta}.
	\end{split}
    \end{equation*}

\end{proof}

\begin{rem}
	We observe that in particular Lemma~\ref{lem-acotacion-maximal-caracteristicaB} holds for all maximal operators appearing in Theorem~\ref{teo-prin}, Theorem~\ref{teo-tipo debil}. Hence they satisfy inequality~\eqref{max_phi_up_down} for some constants $c_1$, $c_2$, $\sig_1$ and $\sig_2$ when applied to the function $\chi_B$.
\end{rem}

\begin{prop} \label{prop-condicion-integrabilidadlocal}
	Let $p\geq1$ and $\phi$ a Young function. There exists $\theta\geq 0$ such that for any ball $Q=B(x_0,\rho(x_0))$ 
	\begin{equation}\label{eq-CondIntLoc-1}
	\int |f|^p {M}_{\phi}^{\theta}(\chi_Q)<\infty
	\end{equation}
	if and only if there exists $\sigma>0$ such that
	\begin{equation}\label{eq-cond-f-integrloc-Tf}
	\int \frac{|f|^p}{(1+|x|)^{\sigma}}<\infty.
	\end{equation}
	
\end{prop}

\begin{proof}
	Let $p\geq 1$ and $\phi$ a Young function. Let $Q=B(x_0,\rho(x_0))$ a critical ball. It is straightforward that there are constants $c$ and $\tilde{c}$ depending on $x_0$ and $\rho$ such that
	\begin{equation}\label{eq-eqen0}
	\frac{c}{1+\frac{|x-x_0|}{\rho(x_0)}}\leq \frac{1}{1+|x|}\leq \frac{\tilde{c}}{1+\frac{|x-x_0|}{\rho(x_0)}}.
	\end{equation}
Then, the equivalence between conditions~\eqref{eq-CondIntLoc-1} and~\eqref{eq-cond-f-integrloc-Tf} follows from equation~\eqref{eq-eqen0} above and Lemma~\ref{lem-acotacion-maximal-caracteristicaB}.
	
\end{proof}

\begin{thm}\label{teo-ChiB}
	Let $1\leq p<\infty$ and $T$ an operator such that for some Young function $\phi$ and for all $\theta$ there exists a constant $C$ such that
	\begin{equation}\label{eq-teo-ChiB}
	\int |Tf|^p w \leq C\int |f|^p M_{\phi}^{\theta}w,
	\end{equation}
	for any weight $w$. Then, if  a function $f$ satisfy~\eqref{eq-cond-f-integrloc-Tf}, $Tf\in L^{p}_{loc}$. In particular $Tf$ is finite almost everywhere.
\end{thm}
\begin{proof}
	Let $1\leq p < \infty$ and $T$ as stated. Let $f$ be a function satisfying~\eqref{eq-cond-f-integrloc-Tf} for some $\sigma>0$.  Then, applying Proposition~\ref{prop-condicion-integrabilidadlocal}, there exists some $\theta\geq 0$ such that~\eqref{eq-CondIntLoc-1} holds for any  critical ball $Q$. 
	
	Let $B$ be a ball in $\RR^d$. According to Proposition~\ref{prop-cubrimientocritico} we can cover $B$ by a finite number of critical balls $B_1,\dots  B_N$. Using the hypothesis on the operator for such $\theta$, 
	\begin{equation*}
	\begin{split}
	\int_B |Tf|^p & \leq \sum_{i=1}^{N}  \int |Tf|^p \chi_{B_i}
	\\ & \leq  C \sum_{i=1}^{N}  \int |f|^p M_{\phi}^{\theta}\chi_{B_i}<\infty.
	\end{split}
	\end{equation*}
	
\end{proof}

For operators that satisfy a weak type inequality for $p=1$ we obtain an analogous result following the same lines as in the proof of Theorem~\ref{teo-ChiB}.

\begin{thm}\label{teo-ChiBdebil}
	Let  $T$ be an operator such that for some Young function $\phi$ and for all $\theta$ there exists a constant $C$ such that
	\begin{equation*}
	w(\{|Tf|>\lambda\}) \leq C \int |f| M_{\phi}^{\theta}w,\text{   for all   } \lambda>0,
	\end{equation*}
	for all weight $w$. Then, if  a function $f$ satisfy~\eqref{eq-cond-f-integrloc-Tf} with $p=1$, $Tf\in L^{1, \infty}_{loc}$. In particular $Tf$ is finite almost everywhere.
\end{thm}

The above results can be applied to all operators considered in Section~\ref{sec-aplSchrodinger} since, as it was shown there, theorems of Section~\ref{sec-general} hold in those cases. In particular we point out that for $\mathcal{R}_1$ and $\mathcal{R}_1^{\star}$ we can apply Theorem~\ref{teo-ChiB}, for $1<p<\infty$, and Theorem~\ref{teo-ChiBdebil}, if $V\in RH_q$ with $q>d$. As for the case $d/2<q<d$, the conclusion  holds for $1<p<p_0$ and $p>p_0'$ respectively. On the other hand, Theorem~\ref{teo-ChiB} and Theorem~\ref{teo-ChiBdebil} can be applied to $\mathcal{R}_2$ for $1<p<q$, when $q>d/2$.

Similarly $VL^{-1}$, $V^{1/2}L^{-1}$ and $V^{1/2}L^{-1/2}$ fall under the scope of Theorem~\ref{teo-ChiB} for $1\le p < q$, $1\leq p \leq p_1$ and $1\le p < 2q$, respectively (see Theorem~\ref{teo:Vcongamma} and Theorem~\ref{teo:mixtos}).

In~\cite{shen}, Shen obtained $L^p$-estimates for derivatives of solutions of differential equations related to Schr\"odinger operator as a consequence of $L^p$-continuity of Riesz-Schr\"odinger Transforms (see Corollary~$0.9$ and Corollary~$0.10$). Here, with our results, we can give qualitative information on their local integrability.

\begin{cor}
	Suppose $V\in RH_q$ for some $q>d/2$. Assume that $-\Delta u + Vu =f$ in $\RR^d$, with $f$ satisfying~\eqref{eq-cond-f-integrloc-Tf} for some $\sigma>0$ and some $p\geq 1$. Then,
	\begin{enumerate}
		 \item if $1<p<q$, $\nabla^2 u\in L^p_{loc}$,
		 \item if $1\leq p<q$, $V u\in L^p_{loc}$ ,
		 \item if $1\leq p< p_1$, $V^{1/2}\nabla u\in L^p_{loc}$,
	\end{enumerate}
	with $p_1$ such that $1/p_1 = \left(1/q-1/d\right)^+ + 1/2q$. 
\end{cor}
\begin{proof}
	We just apply Theorem~\ref{teo-ChiB} to the operators $\nabla^2 L^{-1}$, $V L^{-1}$ and $V^{1/2}\nabla L^{-1}$.
	
\end{proof}

\begin{cor}
	Suppose $V\in RH_q$ for some $q>d/2$ and let $p_0'<p<p_0$, with $p_0$ such that $1/p_0 = \left(1/q-1/d\right)^+$. Assume that $-\Delta u + Vu =\nabla \cdot F$ in $\RR^d$, for a field $F$ with $|F|$ satisfying~\eqref{eq-cond-f-integrloc-Tf} for some $\sigma>0$.		\begin{enumerate}
			\item\label{item-aux1} If $p_0'<p<p_0$, then
			\item\label{item-aux2} If $p_0'<p<2q$, then $V^{1/2} u\in L^p_{loc}$.
		\end{enumerate}
\end{cor}

\begin{proof}
	We will show only item~\eqref{item-aux1}. The proof of~\eqref{item-aux2} is similar. Let $u=L^{-1}\nabla\dot F$. Then $\nabla u = \mathcal{R}_1(\mathcal{R}_1^\star\cdot F)$. Then in order to get that $\nabla u\in L^p_{\text{loc}}$ (due to Theorem~\ref{teo-ChiB}) it will be enough to check that the operators $T_j = \mathcal{R}_1 \circ (\mathcal{R}_1^\star)_j$ satisfy inequality \eqref{eq-teo-ChiB}. In fact, if $p_0'<p<p_0$, then
	\begin{equation}
	\begin{split}
	\int |T_jf|^pw
	& \lesssim \int |(\mathcal{R}_1^{\star})_jf|^{p} M_r^{\theta}w
	\\ & \lesssim \int |f|^{p} M_\nu^{\theta}M_r^{\theta}w.
	\end{split}
	\end{equation}
	for any $\nu>1$. Choosing $\nu>r$, it follows easily  $M_\nu^\theta( M_r^{\theta}w) \leq M_\nu^{\theta}w$, and then \eqref{eq-teo-ChiB} holds.
	
\end{proof}

\bibliographystyle{plain}

\bibliography{analysis}

\end{document}